\documentclass[preprint,authoryear,12pt]{elsarticle}

\usepackage[cp1251]{inputenc}
\usepackage[english]{babel}
\usepackage{colortbl}
\usepackage{xcolor}

 \usepackage{graphicx}

\def\R{\mathbb R}

\def\N{\mathbb N}

\usepackage{amssymb}
\usepackage[intlimits]{amsmath}
 \usepackage{amsthm}

\newtheorem{thm}{Theorem}
\newtheorem{lem}[thm]{Lemma}
\newtheorem{Definition}[thm]{Definition}
\newdefinition{rmk}{Remark}
\newproof{pf}{Proof}
\newproof{pot}{Proof of Theorem \ref{FD_Theor}}

 \biboptions{longnamesfirst,comma}

 \biboptions{square}

\journal{Applied Mathematics and Computation}

\begin{document}

\begin{frontmatter}



\title{A numeric-analytical method for solving the Cauchy problem for ordinary differential equations.}


\author{Volodymyr Makarov\fnref{fn1}}
\ead{makarov@imath.kiev.ua}
\author{Denis Dragunov\corref{cor1}\fnref{fn1}}
\ead{dragunovdenis@gmail.com}
\cortext[cor1]{Corresponding author}
\fntext[fn1]{Institute of Mathematics NAS of Ukraine}

\address{Institute of Mathematics NAS of Ukraine,

Terezhenkivska street 3, Kyiv, Ukraine, Postal Code: 01601

Tel: +380 (44) 234 5150

Fax: +380(44) 235 2010 }

\begin{abstract}
   In the paper we offer a functional-discrete method for solving the Cauchy problem for the first order ordinary differential equations (ODEs). This method (FD-method) is in some sense similar to the Adomian Decomposition Method. But it is shown that for some problems FD-method is convergent whereas ADM is divergent. The results presented in the paper can be easily generalized on the case of systems of ODEs.
\end{abstract}

\begin{keyword}
FD-method \sep Adomian decomposition method \sep generating function method \sep ordinary differential equation\sep Cauchy problem\sep exponential convergence rate.

\MSC  65L05\sep 65L20\sep 65L80
\end{keyword}
\end{frontmatter}



   \section{Introduction.}
 Many scientific papers devoted to the Adomian decomposition method (ADM) have been published during last two decades. For the first time this method was proposed by the American physicist G. Adomian  as a method for solving an operator equations (see \cite{Adomian_1}-\cite{Adomian_3}). The method is based on the specific analytical representation of the exact solution. More precisely, the solution is represented in terms of a rapidly convergent infinite series with easily computable terms (see \cite{Yves_Cherruault}--\cite{Ibrahim_El-Kalla} and the references therein). In spite of the considerable scientific activity, the necessary and sufficient conditions that provide a convergence of these series are unknown.
 Partially, this topic was discussed in \cite{Yves_Cherruault} -- \cite{Hosseini_Nasabzadeh}, where authors study the question about convergence of ADM applied to the nonlinear operator equation $y-N\left(y\right)=f$, where $N\left(\cdot\right)$ is a nonlinear operator in {a} Hilbert space $H$. The paper \cite{N_Himoun_K_Abbaoui_Y_Cherruault} offers sufficient conditions that provide a convergence of ADM applied to the Cauchy problem on a finite segment. {In} \cite{Hashim_Noorani_Ahmad_Bakar_Ismail_Zakaria} authors apply ADM (in a rather modified form) to the Lorenz system. The paper \cite{Inc_Cherruault} is devoted to the application of ADM to the solution of the nonlinear Volterra-Fredholm integro-differential equation.

 The author of \cite{Ibrahim_El-Kalla} {proposed} a modification of the ADM. Using several numerical examples, he has shown that this modified method converges faster than the ordinary one. {Nevertheless,  one crucial fact was overlooked} in his work, {as it turns out,} this modified ADM coincides with the fixed point iteration method (c.f. \cite{Gavrilyuk_Lazurchak_Makarov_Sytnik}).

The idea of ADM is similar to the idea of a functional-discrete method, called FD-method, which was firstly introduced in \cite{Makarov_SLP}. In this paper the FD-method was applied to the solution of the Sturm-Liouville problem and the remarkable convergence results were obtained. Then, in \cite{GKMR} --
\cite{Gavrilyuk_Lazurchak_Makarov_Sytnik}, {FD-method was  successfully applied} to several operator
equations and, in particular, to the boundary value problems.

The essential difference between ADM and FD-method {is expressed by the fact that} the last one has a built-in adjustable parameter, by varying which we can provide the convergence of FD-method {even}, when ADM is found to be divergent.

Let us outline the general idea of FD-method by applying it to the following Cauchy problem:
\begin{equation}\label{zagalna_zadacha}
    \begin{array}{c}
      \mathbb{L}_{r}\left(u\left(x\right)\right)-N\left(x,
    u\left(x\right)\right)u\left(x\right)=\phi\left(x\right),
    x\in\left[x_{0}, +\infty\right), \\[1.2em]
      u\left(x_{0}\right)=u_{0},\frac{d}{dx}u\left(x_{0}\right)=u^{(1)}_{0},\ldots, \frac{d^{r-1}}{dx^{r-1}}u\left(x_{0}\right)=u^{(r-1)}_{0},
    \end{array}
\end{equation}
where $\mathbb{L}_{r}\left(\cdot\right)=\frac{d^{r}}{dx^{r}}\left(\cdot\right)+\sum\limits_{k=1}^{r-1}p_{k}\left(x\right)\frac{d^{\left(r-k\right)}}{dx^{\left(r-k\right)}}\left(\cdot\right)$
is a linear differential operator of the $r-$th order, $r\in
\mathbb{N}, p_{k}\left(x\right)\in C\left(\left[x_{0},
+\infty\right)\right);$
$N\left(x, u\right)$: $\left[x_{0}, +\infty\right)\times \R
\rightarrow \mathbb{R}$ is a given nonlinear function, $\phi\left(x\right)\in C\left(\left[x_{0}, +\infty\right)\right).$  From now on we make the assumption that {$N\left(x, u\right)$} is continuous in $x$ on
$\left[x_{0}, +\infty\right)$ and infinitely differentiable with respect to $u$ on $\R,$ i.e., $N\left(x, u\right)\in C^{0, \infty}_{x, u}\left(\left[x_{0}, +\infty\right)\times \R\right).$

{
The main idea of FD-method is to find an approximation to the exact solution $u\left(x\right)$ of problem (\ref{zagalna_zadacha}) in the form of the finite subsum $\overset{m}{u}\left(x\right)=\sum\limits_{i=0}^{m}u^{\left(i\right)}\left(x\right)$ of the infinite series representation
\begin{equation}\label{ZagalniyRyad}
u\left(x\right)=
\sum_{i=0}^{\infty}u^{\left(i\right)}\left(x\right).
\end{equation}
To define each term of series (\ref{ZagalniyRyad}) we need to introduce a grid:
\begin{equation}\label{grid}
    \begin{array}{c}
      \widehat{\omega}=\left\{x_{0}<x_{1}<x_{2}<\ldots,\quad x_{n}\rightarrow +\infty, \; n\rightarrow +\infty\right\}, \\[1.2em]
      h=\sup\limits_{i\in \N}\left\{h_{i}=x_{i}-x_{i-1}\right\}.
    \end{array}
\end{equation}
We define the functions
$u^{\left(i\right)}\left(x\right) \in
C^{r-1}\left(\left[x_{0}, +\infty\right)\right)$ to be the solutions of the following linear Cauchy problems
\begin{equation}\label{linear_Caucy_problems_1}
    \begin{array}{c}
      \mathbb{L}_{r}\left(u^{(0)}\left(x\right)\right)-N\left(x,\; u^{(0)}\left(x_{i-1}\right)\right)u^{(0)}\left(x\right)=\phi\left(x\right), \\[1.2em]
      u^{(0)}\left(x_{0}\right)=u^{(0)}_{0}, \frac{d}{dx}u\left(x_{0}\right)=u^{(1)}_{0},\ldots, \frac{d^{r-1}}{dx^{r-1}}u\left(x_{0}\right)=u^{(r-1)}_{0}, \\[1.2em]
      x\in\left[x_{i-1},\; x_{i}\right],\; i=1,2,\ldots,
    \end{array}
\end{equation}
\begin{equation}\label{linear_Caucy_problems_2}
    \begin{array}{c}
      \mathbb{L}_{r}\left(u^{(j+1)}\left(x\right)\right)-N\left(x,\;u^{(0)}\left(x_{i-1}\right)\right)u^{(j+1)}\left(x\right)= \\[1.2em]
      =N'\left(x,\;u^{(0)}\left(x_{i-1}\right)\right)u^{(0)}\left(x\right)u^{(j+1)}\left(x_{i-1}\right)+F^{(j+1)}\left(x\right), \\[1.2em]
      u^{(j+1)}\left(x_{0}\right)=0, \frac{d}{dx}u^{(j+1)}\left(x_{0}\right)=0,\ldots, \frac{d^{r-1}}{dx^{r-1}}u^{(j+1)}\left(x_{0}\right)=0, \\[1.2em]
      x\in\left[x_{i-1},\; x_{i}\right], \; i=1,2,\ldots
    \end{array}
\end{equation}
with {the matching conditions}
\begin{equation}\label{matching_conditions}
    \begin{array}{c}
      \left[\frac{d^{k}}{d x^{k}}u^{(j)}\left(x\right)\right]_{x=x_{i}}=\lim\limits_{x\rightarrow x_{i}+0}\frac{d^{k}}{d x^{k}}u^{(j)}\left(x\right)-\lim\limits_{x\rightarrow x_{i}-0}\frac{d^{k}}{d x^{k}}u^{(j)}\left(x\right)=0, \\[1.2em]
      k=0,1,\ldots, r-1;\;i=1,2,\ldots; j=0,1,\ldots,
    \end{array}
\end{equation}

where
\begin{equation}\label{F_term}
    \begin{array}{c}
      F^{(j+1)}\left(x\right)=\sum\limits_{p=1}^{j}A_{j+1-p}\left(N\left(x,\;\left(\cdot\right)\right);\; u^{(0)}\left(x_{i-1}\right), \ldots , u^{(j+1-p)}\left(x_{i-1}\right)\right)u^{(p)}\left(x\right)+ \\[1.2em]
      +\sum\limits_{p=0}^{j}\left[A_{j-p}\left(N\left(x,\;\left(\cdot\right)\right);\; u^{(0)}\left(x\right), u^{(1)}\left(x\right), \ldots, u^{(j-p)}\left(x\right) \right)-\right. \\[1.2em]
      \left.-A_{j-p}\left(N\left(x,\;\left(\cdot\right)\right);\; u^{(0)}\left(x_{i-1}\right), u^{(1)}\left(x_{i-1}\right), \ldots, u^{(j-p)}\left(x_{i-1}\right) \right)\right]u^{(p)}\left(x\right)+ \\[1.2em]
      +A_{j+1}\left(N\left(x,\;\left(\cdot\right)\right);\; u^{(0)}\left(x_{i-1}\right), \ldots, u^{(j)}\left(x_{i-1}\right), 0\right)u^{(0)}\left(x\right),\; j=0,1,\ldots.
    \end{array}
\end{equation}
And
\begin{equation}\label{Ado_poly}
      A_{k}\left(N\left(x,\; \left(\cdot\right)\right);\;u^{(0)}, u^{(1)},\ldots, u^{(k)}\right)=\left.\frac{1}{k!}\frac{d^{k}}{d t^{k}}\left(N\left(x,\;\sum_{i=0}^{\infty}t^{i}u^{(i)}\right)\right)\right|_{t=0}
\end{equation}
{are well-known Adomian's polynomials} (see \cite{Seng_Abbaoui_Cherruault}, \cite{Seng_Abbaoui_Cherruault_1}).

\begin{Definition}\label{defin}
  We say that the FD-method for the Cauchy problem (\ref{zagalna_zadacha}) converges (to the exact solution of problem (\ref{zagalna_zadacha})) on $\left[x_{0}, x_{0}+H\right),\; 0<H\leq+\infty,$ if there exists a real number $\overline{h}>0$ such that for any grid $\omega$ (\ref{grid}) with $h\leq\overline{h},$ series (\ref{ZagalniyRyad}), with the terms computed from (\ref{linear_Caucy_problems_1})--(\ref{F_term}), converges uniformly (to the exact solution of problem (\ref{zagalna_zadacha})) on $\left[x_{0}, x_{0}+H\right).$
\end{Definition}

In the present paper we consider a particular case of the Cauchy problem (\ref{zagalna_zadacha}) when $r=1, \mathbb{L}_{1}\left(\cdot\right)=\frac{d}{dx}\left(\cdot\right),$ more precisely
\begin{equation}\label{OFD_1}
   \frac{d}{dx}u\left(x\right)-N\left(x,\;u\left(x\right)\right)u\left(x\right)=\phi\left(x\right),
\end{equation}
$$u\left(x_{0}\right)=u_{0},\; x \in \left[x_{0}, +\infty\right).$$
The main result of the paper is presented by the following theorem:
                    \begin{thm}\label{FD_Theor}

                    Let the Cauchy problem (\ref{OFD_1}) satisfies the following conditions:
\begin{enumerate}
                        \item $N\left(x,\;u\right)=\sum\limits_{i=0}^{\infty}a_{i}\left(x\right)u^{i},$ where $a_{i}\left(x\right)\in C\left(\left[x_{0},\;+\infty\right)\right).$ And there exists a sequence of real numbers $B_{i}>0,\ldots i=0,1,\ldots$, such that $$\sup\limits_{x\in \left[x_{0},+\infty\right)}\left|a_{i}\left(x\right)\right|\leq B_{i},  \;0\leq B_{i}\in \R,\; i=0,1,\ldots,$$ and the series $\sum\limits_{i=0}^{\infty}B_{i}u^{i}$ is convergent  $\forall u\in\R;$
                        \item $\phi\left(x\right)$ is a continuous and bounded function on $\left[x_{0}, +\infty\right)$, with $$\sup\limits_{x\in \left[x_{0},\;+\infty\right)}\left|\phi\left(x\right)\right|=k<+\infty;$$

                        \item  $\left(N\left(x,\;u\right)u\right)^{\prime}_{u}=N\left(x,\;u\right)+uN^{\prime}_{u}\left(x,\;u\right)< -\alpha,<0,$ $\quad \forall x\in \left[x_{0},+\infty\right),$ $ \forall u\in \mathbb{R}.$
\end{enumerate}
                        Then for any initial condition $u_{0}\in \mathbb{R}$ the solution
                        $u\left(x\right)$ of problem (\ref{OFD_1}) exists on $\left[x_{0}, +\infty \right).$ The FD-method converges on $\left[x_{0}, +\infty \right)$ to the exact solution of problem (\ref{OFD_1}). Moreover, the following error estimations hold true
                        \begin{equation}\label{error_estimations1}
                              \sup\limits_{x\in\left[0, +\infty\right)}\left|u\left(x\right)-\overset{m}{u}\left(x\right)\right|\leq
                         \frac{C}{(1+m)^{1+\varepsilon}} \frac{(h/R)^{m+1}}{1-h/R}\quad \mbox{if} \quad h<R,
                        \end{equation}
                        \begin{equation}\label{error_estimations2}
                              \sup\limits_{x\in\left[0, +\infty\right)}\left|u\left(x\right)-\overset{m}{u}\left(x\right)\right|\leq C\sum\limits_{j=m+1}^{\infty}\frac{1}{\left(j+1\right)^{1+\varepsilon}}\quad \mbox{if}\quad h=R,
                        \end{equation}
                        were $\varepsilon ,\; R, C$ are positive real numbers that depend on the input data of the problem only.
                    \end{thm}

\section{Justification of the FD-method algorithm for solving the Cauchy problem on the infinite interval.}
To begin with, let us introduce some useful notations.
For each real-valued function $u\left(x\right)$ defined on $\left<a,\; b\right>, \; a<b\leq+\infty$  we denote
$$\left\|u\left(x\right)\right\|_{0,\infty,\;\left< a,\;b\right>}=\sup_{x\in \left<a,\;b\right>}\left|u\left(x\right)\right|,$$ and $\forall u\left(x\right)\in C\left(\left<a, b\right>\right)$ we denote $$\left\|u\left(x\right)\right\|_{1,\infty,\;\left< a,\;b\right>}=$$
$$=\max\left\{\left\|u\left(x\right)\right\|_{0,\infty,\left< a,\;b\right>},\;\left\|\frac{d}{d x}u\left(x\right)\right\|_{0,\infty,\left< a,\;b\right>}\right\}.$$
For the future convenience we would like to define such notations
\begin{equation}\label{notation_1}
    \left\|u\left(x\right)\right\|_{i,\;\left< a,\;b\right>}\overset{def}{=}\left\|u\left(x\right)\right\|_{i,\infty,\;\left< a,\;b\right>},\;\left\|u\left(x\right)\right\|_{i}\overset{def}{=}\left\|u\left(x\right)\right\|_{i,\infty,\;\left[ x_{0},\;+\infty\right)},\;i=1,2.
\end{equation}
Before proceed to the proof of theorem \ref{FD_Theor} we need to prove several auxiliary statements, presented below.

                   \begin{lem}\label{lema_pro_obmezh}
                      Let conditions 2) and 3) of theorem \ref{FD_Theor} be fulfilled, and
                      $N\left(x,\;u\right)\in C^{0,1}_{x,u}\left(\left[x_{0},\;+\infty\right)\times\R\right).$ Then for any $h_{i}>0,\; i=1,2,\ldots$ the solution $u^{(0)}\left(x\right)$ of the following Cauchy problem with piecewise constant argument (see. \cite{Akhmet_PC_argument})
                     \begin{equation}\label{dacha1}
                        \frac{d}{dx}u^{(0)}\left(x\right)-N\left(x,\;u^{(0)}\left(x_{i-1}\right)\right)u^{(0)}\left(x\right)=\phi\left(x\right), \; x\in\left[x_{i-1}, \; x_{i}\right], \; x_{i}=x_{i-1}+h_{i},
                     \end{equation}
                     \begin{equation}\label{dacha2}
                        u^{(0)}\left(x_{0}\right)=u_{0}\in \R, \;\left[u^{(0)}\left(x\right)\right]_{x=x_{i}}=0,\;i=1,2,\ldots
                        \end{equation}
                      exists on $\left[x_{0}, +\infty\right)$ and $\left|u^{(0)}\left(x\right)\right|\leq \mu=\max\left\{\left|u_{0}\right|,\; \frac{k}{\alpha}\right\},\;\forall x\in\left[x_{0},\;+\infty\right).$
                   \end{lem}
                  \begin{proof}
                   Since equation (\ref{dacha1}) is linear on each segment $\left[x_{i-1},\;
                   x_{i}\right]$ and $N\left(x,\;u\right)$ is defined on the whole real line as a function of $u$, the solution of problem (\ref{dacha1}), (\ref{dacha2}) on the interval $\left[x_{0},\; +\infty\right)$ exists obviously:
                    \begin{equation}\label{dacha3}
                            \begin{array}{c}
                            u^{(0)}\left(x\right)=\exp\left\{\int\limits_{x_{i-1}}^{x}N\left(\xi,u^{(0)}\left(x_{i-1}\right)\right)d\xi\right\}u^{(0)}\left(x_{i-1}\right)+\\
                            +\int\limits_{x_{i-1}}^{x}\exp\left\{\int\limits_{\xi}^{x}N\left(\tau,u^{(0)}\left(x_{i-1}\right)\right)d\tau\right\}\phi\left(\xi\right)d\xi,\quad x\in\left[x_{i-1},\;x_{i}\right].
                            \end{array}
                     \end{equation}
                    Thus, it remains only to prove that this solution is bounded by some positive constant $\mu.$ The conditions of this lemma together with the mean value theorem provide us with the following inequality
                   $$N\left(x,\;u\right)u^{2}=u\left[N\left(x,\;u\right)u-N\left(x,\;0\right)\cdot 0\right]=$$
                   $$=u\left[\left(N_{u}^{\prime}\left(x,\;\tilde{u}\right)\tilde{u}+N\left(x,\;\tilde{u}\right)\right)u\right]\leq -\alpha u^{2},\; \tilde{u}\in\left[0,\; u\right].$$
                   If we, additionally, take into account that the function $N\left(x,\;u\right),$ is continuous, we will obtain the inequality $N\left(x,\;u\right)\leq -\alpha,\; \forall u\in \R.$
                   Then multiplying both sides of equation (\ref{dacha1}) (for $i=1$) by $u^{(0)}\left(x\right)$ and using the above estimate for
                   $N\left(x,\;u\right)$ we get the following
                     $$\frac{1}{2}\frac{d}{d x}\left(u^{(0)}\left(x\right)\right)^{2}=N\left(x, u^{(0)}\left(x_{0}\right)\right)\left(u^{(0)}\left(x\right)\right)^{2}+\phi\left(x\right)u^{(0)}\left(x\right),$$
                     \begin{equation}\label{dacha1_1}
                        \frac{1}{2}\frac{d}{d x}\left(u^{(0)}\left(x\right)\right)^{2}\leq -\alpha\left(u^{(0)}\left(x\right)\right)^{2}+k \left|u^{(0)}\left(x\right)\right|.
                     \end{equation}
                      Inequality (\ref{dacha1_1}) states that solution $u^{(0)}\left(x\right)$ of problem (\ref{dacha1}), (\ref{dacha2}) on $\left[x_{0}, \;x_{1}\right]$ can't leave a segment $\left[-\mu,\;\mu\right].$

                     As a matter of fact, if we suppose that there exist $x_{\ast}, x^{\ast}\in \left[x_{0}, x_{1}\right]$ such that $x_{\ast}<x^{\ast}$ and
                     $$\left|u^{(0)}\left(x_{\ast}\right)\right|=\mu\geq\frac{k}{\alpha},$$
                     \begin{equation}\label{_pt_1}
                     \left|u^{(0)}\left(x\right)\right|>\mu,\; \forall x\in\left(x_{\ast}, x^{\ast}\right],
                     \end{equation}
                     then  we immediately obtain from (\ref{dacha1_1})
                     $$\left.\frac{d}{d x}\left(\left(u^{(0)}\left(x\right)\right)^{2}\right)\right|_{x=x^{\ast}}<0.$$
                     This contradicts our assumption \eqref{_pt_1}.
                     So, $u^{(0)}\left(x\right)\in \left[-\mu,\;\mu\right],\; \forall x\in\left[x_{0},\; x_{1}\right].$
                     In a similar way we can prove that the solution of the Cauchy problem  (\ref{dacha1}), (\ref{dacha2}) on $\left[x_{1}, x_{2}\right]$
                     belongs to the segment $\left[-\mu,\;\mu\right]$ too and so on. This completes the proof of the lemma.
                    \end{proof}
The following lemma can be considered as a generalization of lemma 2.1 from \cite{Gavrilyuk_Lazurchak_Makarov_Sytnik}.
\begin{lem}\label{lema_1}
Let $N\left(x,u\right)\in
C^{0,\infty}_{x,u}\left(\left[x_{0},\;+\infty\right)\times\R\right)$ and
$u^{(j)}\left(x\right)\in C^{1}\left(\left[x_{0},\; +\infty\right)\right),$
$ j=0,1,2,\ldots.$
Furthermore, let there exists a function
$\widetilde{N}\left(u\right)\in C^{\infty}\left(\R\right)$ such that
$$\sup_{x\in\left[x_{0},\;+\infty\right)}\left|\frac{\partial^{p} }{\partial u^{p}}N\left(x,u\right)\right|\leq \frac{d^{p}}{d u^{p}}\widetilde{N}\left(\left|u\right|\right)\; \forall u\in\R,\; \forall p\in \N\bigcup \left\{0\right\}.$$
Then the following inequalities hold true
$$\left\|A_{k}\left(N\left(x,\left(\cdot\right)\right);\;u^{(0)}\left(x\right),\ldots, u^{(k)}\left(x\right)\right)-\right.$$
$$\left.-A_{k}\left(N\left(x,\left(\cdot\right)\right);\;u^{(0)}\left(x_{i-1}\right),\ldots, u^{(k)}\left(x_{i-1}\right)\right)\right\|_{0,\left[x_{i-1},\; x_{i} \right]}\leq$$
$$\leq h_{i}A_{k}\left(\widetilde{N}'\left(u\right)u;\;\left\|u^{(0)}\left(x\right)\right\|_{1,\left[x_{i-1},\;x_{i}\right]},\ldots, \left\|u^{(k)}\left(x\right)\right\|_{1,\left[x_{i-1},\;x_{i}\right]}\right)\leq$$
$$\leq h_{i}A_{k}\left(\widetilde{N}'\left(u\right)u;\;\left\|u^{(0)}\left(x\right)\right\|_{1,\left[x_{0},+\infty\right)},\ldots, \left\|u^{(k)}\left(x\right)\right\|_{1,\left[x_{0},+\infty\right)}\right),$$
$$\forall x\in \left[x_{i-1}, x_{i}\right], \; k=0,1,2,\ldots, i=1,2,\ldots.$$
\end{lem}
\begin{proof} Throughout the proof we use the notation
$$N^{(p)}\left(u\right)=\frac{\partial^{p} }{\partial u^{p}}N\left(x,u\right),\quad u^{(j)}=u^{(j)}\left(x\right),\quad u^{(j)}_{i-1}=u^{(j)}\left(x_{i-1}\right).$$
Let us fix any arbitrary $x\in\left[x_{i-1}, x_{i}\right].$ It is well known that the following representation for the Adomian's polynomial holds
(see \cite{Abbaoui_Pujol_Cherruault_Himoun_Grimalt},
\cite{Seng_Abbaoui_Cherruault}):
$$A_{k}\left(N\left(\cdot\right);\; u^{(0)},\ldots, u^{(k)}\right)=$$
$$=\sum_{\substack{\alpha_{1}+\ldots+\alpha_{k}=k \\ \alpha_{1}\geq \alpha_{2}\geq\ldots\geq\alpha_{k},\; \alpha_{i}\in \N\bigcup \left\{0\right\},}}N^{\left(\alpha_{1}\right)}\left(u^{(0)}\right)\frac{\left[u^{(1)}\right]^{\left(\alpha_{1}-\alpha_{2}\right)}}{\left(\alpha_{1}-\alpha_{2}\right)!}\ldots
\frac{\left[u^{(k-1)}\right]^{\left(\alpha_{k-1}-\alpha_{k}\right)}}{\left(\alpha_{k-1}-\alpha_{k}\right)!}\frac{\left[u^{(k)}\right]^{\alpha_{k}}}{\left(\alpha_{k}\right)!}.$$
Therefore
$$\left\|A_{k}\left(N\left(\cdot\right);\; u^{(0)},\ldots, u^{(k)}\right)-A_{k}\left(N\left(\cdot\right);\; u^{(0)}_{i-1},\ldots, u^{(k)}_{i-1}\right)\right\|_{0,\left[x_{i-1},\; x_{i}\right]}\leq$$
$$\leq\left\|\sum_{\alpha_{1}+\ldots+\alpha_{k}=k}\left\{N^{\left(\alpha_{1}\right)}\left(u^{(0)}\right)\frac{\left[u^{(1)}\right]^{\left(\alpha_{1}-\alpha_{2}\right)}}{\left(\alpha_{1}-\alpha_{2}\right)!}\ldots
\frac{\left[u^{(k)}\right]^{\alpha_{k}}}{\left(\alpha_{k}\right)!}-\right.\right.$$
$$-\left.\left.N^{\left(\alpha_{1}\right)}\left(u^{(0)}_{i-1}\right)\frac{\left[u^{(1)}_{i-1}\right]^{\left(\alpha_{1}-\alpha_{2}\right)}}{\left(\alpha_{1}-\alpha_{2}\right)!}\ldots
\frac{\left[u^{(k)}_{i-1}\right]^{\alpha_{k}}}{\left(\alpha_{k}\right)!}\right\}\right\|_{0,\left[x_{i-1},
\;x_{i}\right]}=$$
$$=\left\|h^{\ast}\sum_{\alpha_{1}+\ldots+\alpha_{k}=k}\left\{N^{\left(\alpha_{1}+1\right)}\left(\widetilde{u}^{(0)}\right)\left.\frac{d u^{(0)}\left(x\right)}{dx}\right|_{x=\widetilde{x}}\frac{\left[\widetilde{u}^{(1)}\right]^{\left(\alpha_{1}-\alpha_{2}\right)}}{\left(\alpha_{1}-\alpha_{2}\right)!}\ldots
\frac{\left[\widetilde{u}^{(k)}\right]^{\alpha_{k}}}{\left(\alpha_{k}\right)!}+\right.\right.$$
$$\left.\left.+\ldots+N^{\left(\alpha_{1}\right)}\left(\widetilde{u}^{(0)}\right)\frac{\left[\widetilde{u}^{(1)}\right]^{\left(\alpha_{1}-\alpha_{2}\right)}}{\left(\alpha_{1}-\alpha_{2}\right)!}\ldots
\frac{\left[\widetilde{u}^{(k)}\right]^{\alpha_{k}-1}}{\left(\alpha_{k}-1\right)!}\left.\frac{d
u^{(k)}\left(x\right)}{dx}\right|_{x=\widetilde{x}}\right\}\right\|_{0,\left[x_{i-1}, \;x_{i}\right]}\leq$$

$$\leq h_{i}\sum_{\alpha_{1}+\ldots+\alpha_{k}=k}\frac{\left\|u^{(1)}\right\|_{1,\left[x_{i-1},\;x_{i}\right]}^{\alpha_{1}-\alpha_{2}}}{\left(\alpha_{1}-\alpha_{2}\right)!}\ldots \frac{\left\|u^{(k-1)}\right\|_{1,\left[x_{i-1},\;x_{i}\right]}^{\alpha_{k-1}-\alpha_{k}}}{\left(\alpha_{k-1}-\alpha_{k}\right)!}\frac{\left\|u^{(k)}\right\|_{1,\left[x_{i-1},\;x_{i}\right]}^{\alpha_{k}}}{\left(\alpha_{k}\right)!}\times$$
$$\times \left\{\left\|N^{\left(\alpha_{1}+1\right)}\left(u^{(0)}\right)u^{(0)}\right\|_{0,\left[x_{i-1},\; x_{i}\right]}+\alpha_{1}\left\|N^{\left(\alpha_{1}\right)}\left(u^{(0)}\right)\right\|_{0,\left[x_{i-1},\;x_{i}\right]}\right\}\leq$$

$$\leq h_{i}\sum_{\alpha_{1}+\ldots+\alpha_{k}=k}\frac{\left\|u^{(1)}\right\|_{1,\left[x_{i-1},\;x_{i}\right]}^{\alpha_{1}-\alpha_{2}}}{\left(\alpha_{1}-\alpha_{2}\right)!}\ldots \frac{\left\|u^{(k-1)}\right\|_{1,\left[x_{i-1},\;x_{i}\right]}^{\alpha_{k-1}-\alpha_{k}}}{\left(\alpha_{k-1}-\alpha_{k}\right)!}\frac{\left\|u^{(k)}\right\|_{1,\left[x_{i-1},\;x_{i}\right]}^{\alpha_{k}}}{\left(\alpha_{k}\right)!}\times$$
$$\times \left\{\widetilde{N}^{\left(\alpha_{1}+1\right)}\left(\left\|u^{(0)}\right\|_{1,\left[x_{i-1},\; x_{i}\right]}\right)\left\|u^{(0)}\right\|_{1,\left[x_{i-1},\; x_{i}\right]}+\alpha_{1}\widetilde{N}^{\left(\alpha_{1}\right)}\left(\left\|u^{(0)}\right\|_{1,\left[x_{i-1},\; x_{i}\right]}\right)\right\}=$$

$$= h_{i}A_{k}\left(\widetilde{N}'\left(u\right)u;\;\left\|u^{(0)}\left(x\right)\right\|_{1,\left[x_{i-1},\;x_{i}\right]},\ldots, \left\|u^{(k)}\left(x\right)\right\|_{1,\left[x_{i-1},\;x_{i}\right]}\right)\leq$$
$$\leq h_{i}A_{k}\left(\widetilde{N}'\left(u\right)u;\;\left\|u^{(0)}\left(x\right)\right\|_{1,\left[x_{0},\;+\infty\right)},\ldots, \left\|u^{(k)}\left(x\right)\right\|_{1,\left[x_{0},\;+\infty\right)}\right),$$
where $\widetilde{u}^{(j)}=u^{(j)}\left(\widetilde{x}\right),\; \widetilde{x}\in\left[x_{i-1}, x_{i}\right],\; h^{\ast}=x-x_{i-1},$
And this is precisely the assertion of the lemma.
\end{proof}

The following lemma can be considered as a generalization of lemma 2.2 from \cite{Gavrilyuk_Lazurchak_Makarov_Sytnik}.
\begin{lem}\label{lema_2}
For any scalar function $\widetilde{N}\left(u\right)\in C^{\infty}\left(\R\right)$ the following equality holds true
  $$A_{j+1}\left(\widetilde{N}\left(x\right);\;V_{0}, V_{1},\ldots, V_{j}, 0\right)=$$
  $$=\frac{1}{\left(j+1\right)!}\left\{\frac{d^{j+1}}{d z^{j+1}}\left(\widetilde{N}\left(f\left(z\right)\right)-\left(f\left(z\right)-V_{0}\right)\widetilde{N}'\left(V_{0}\right)\right)\right\}_{z=0},$$
  where  $j=0,1,\ldots$, $f\left(z\right)=\sum\limits_{i=0}^{\infty}z^{i}V_{i}.$
\end{lem}
The proof of lemma \ref{lema_2} is trivial.

\begin{proof}{\it (Of theorem \ref{FD_Theor})}
                        Let the conditions of theorem \ref{FD_Theor} be fulfilled. It is easy to seen that for any initial condition $u_{0}\in\R,$
                        the solution of the Cauchy problem (\ref{OFD_1}) exists and is unique on $\left[0, +\infty\right).$ Moreover, it belongs to the segment $\left[-\mu,\;\mu\right],\; \mu=\max\left\{\left|u_{0}\right|,\;\frac{k}{\alpha}\right\}$ (see. \cite{Demidovich}, p.278).
                        So, we are going to prove that this solution can be found by the FD-method.

                        Let us fix any arbitrary infinite grid
                        \begin{equation}\label{Sitka}
                            \widehat{\omega}=\left\{x_{i}=x_{i-1}+h_{i}, \;h_{i}>0,\;i=1,2,\ldots;\; x_{i}\rightarrow +\infty,\; i\rightarrow +\infty\right\},
                        \end{equation}
                        $$h=\sup\limits_{i\in \N}h_{i}$$
                        and embed the Cauchy problem (\ref{OFD_1}) into the more general problem
                        $$\frac{d}{dx}u\left(x,t\right)-N\left(x,u\left(x_{i-1}, t\right)\right)u\left(x, t\right)-$$
                        \begin{equation}\label{OFD_4}
                        -t\left\{N\left(x,u\left(x, t\right)\right)-N\left(x,u\left(x_{i-1},
                        t\right)\right)\right\}u\left(x,t\right)=\phi\left(x\right),\;x\in \left[x_{i-1}, x_{i}\right],
                        \end{equation}
                        $$\left[u\left(x,t\right)\right]_{x=x_{i}}=0,\; i=1,2,\ldots,\;
                        u\left(x_{0},t\right)=u_{0},\; \forall t \in \left[0,1\right].$$
                        Apparently, if we set $t=1$ then the solution of
                        (\ref{OFD_4}) coincides with the solution of problem
                        (\ref{OFD_1}) :
                        $$u\left(x,1\right)=u\left(x\right).$$
                        If we set $t=0,$ on the other hand,  we will get the base problem
                        \begin{equation}\label{OFD_2}
                        \begin{array}{c}
                          \frac{d}{d
                            x}u^{(0)}\left(x\right)-N\left(x,u^{(0)}\left(x_{i-1}\right)\right)u^{(0)}\left(x\right)=\phi\left(x\right),
                            \; x\in\left[x_{i-1},\;x_{i}\right], \; i=1,2,\ldots, \\ [1.2em]
                          u^{(0)}\left(x_{0}\right)=u_{0},\;\left[u^{(0)}\left(x\right)\right]_{x=x_{i}}=0,
                            i=1,2,\ldots,
                        \end{array}
                        \end{equation}
                        which is analogous to problem (\ref{linear_Caucy_problems_1}).
                        Then, we assume that the solution of problem
                        (\ref{OFD_4}) can be found in the form of series
                        \begin{equation}\label{OFD_5}
                            u\left(x,t\right)=\sum_{j=0}^{\infty}t^{j}u^{(j)}\left(x\right),
                             \forall t \in \left[0, 1\right], \forall x \in \left[x_{0},\;+\infty\right),
                        \end{equation}
                        \begin{equation}\label{OFD_5'}
                            \frac{d}{dx}u\left(x,
                            t\right)=\sum_{j=0}^{\infty}t^{j}\frac{d}{dx}u^{(j)}\left(x\right),
                            \forall t\in \left[0, 1\right], \forall
                            x\in \left[x_{i-1},\;x_{i}\right],\;
                            i=1,2,\ldots,
                        \end{equation}
                        where coefficients $u^{(j)}\left(x\right)$ depend on $x$ only.
                        By substituting (\ref{OFD_5}) and (\ref{OFD_5'}) into (\ref{OFD_4}) and comparing the coefficients in front of the powers of $t,$ we obtain the following recurrence sequence of linear problems for $u^{(j+1)}\left(x\right)$ (with a piecewise constant coeficient):
                        \begin{equation}\label{NewDir_11}
                        \begin{array}{c}
                          \frac{d}{dx}u^{(j+1)}\left(x\right)-N\left(x,u^{(0)}\left(x_{i-1}\right)\right)u^{(j+1)}\left(x\right)= \\[1.2em]
                          =N^{\prime}_{u}\left(x,u^{(0)}\left(x_{i-1}\right)\right)u^{(0)}\left(x\right)u^{(j+1)}\left(x_{i-1}\right)+F^{j+1}\left(x\right), \\[1.2em]
                          \quad x\in \left[x_{i-1},x_{i}\right],
                            i=1,2,\ldots,
                        \end{array}
                        \end{equation}
                        where
                        $$\begin{array}{c}
                            F^{(j+1)}\left(x\right)=\sum\limits_{p=1}^{j}A_{j+1-p}\left(N\left(x,\left(\cdot\right)\right); u^{(0)}\left(x_{i-1}\right), \ldots , u^{(j+1-p)}\left(x_{i-1}\right)\right)u^{(p)}\left(x\right)+ \\[1.2em]
                            +\sum\limits_{p=0}^{j}\left[A_{j-p}\left(N\left(x,\left(\cdot\right)\right); u^{(0)}\left(x\right), u^{(1)}\left(x\right), \ldots, u^{(j-p)}\left(x\right) \right)-\right. \\[1.2em]
                            \left.-A_{j-p}\left(N\left(x,\left(\cdot\right)\right); u^{(0)}\left(x_{i-1}\right), u^{(1)}\left(x_{i-1}\right), \ldots, u^{(j-p)}\left(x_{i-1}\right) \right)\right]u^{(p)}\left(x\right)+ \\[1.2em]
                            +A_{j+1}\left(N\left(x,\left(\cdot\right)\right); u^{(0)}\left(x_{i-1}\right), u^{(1)}\left(x_{i-1}\right), \ldots, u^{(j)}\left(x_{i-1}\right), 0\right)u^{(0)}\left(x\right), \\[1.2em]
                            \left[u^{(j+1)}\left(x\right)\right]_{x=x_{i}}=0,\; i=1,2,\ldots,\;
                        u^{(j+1)}\left(x_{0}\right)=0,\; j=0,1,2,\ldots.
                          \end{array}
                        $$ It is easy to see, that problems (\ref{NewDir_11}) are analogous to problems (\ref{linear_Caucy_problems_2}).
Let us express equation (\ref{NewDir_11}) in the equivalent form
\begin{equation}\label{NewDir_7}
\begin{array}{c}
  \frac{d u^{(j+1)}\left(x\right)}{d x}-q_{i}\left(x\right)u^{(j+1)}\left(x\right)= \\[1.2em]
  =N^{\prime}_{u}\left(x,u^{(0)}\left(x_{i-1}\right)\right)\left[u^{(j+1)}\left(x_{i-1}\right)\int\limits_{x_{i-1}}^{x}\frac{d}{d\xi}u^{(0)}\left(\xi\right)d \xi+\right. \\[1.2em]
  \left.+u^{(0)}\left(x_{i-1}\right)\int\limits_{x_{i-1}}^{x}\frac{d}{d\xi}u^{(j+1)}\left(\xi\right)d \xi\right]+F^{(j+1)}\left(x\right),\;x\in\left[x_{i-1},\; x_{i}\right], \\[1.2em]
  \left[u^{(j+1)}\left(x\right)\right]_{x=x_{i}}=0,\; i=1,2,\ldots,\; u^{(j+1)}\left(x_{0}\right)=0,
\end{array}
\end{equation}
where
$$q_{i}\left(x\right)=N\left(x,u^{(0)}\left(x_{i-1}\right)\right)+N^{\prime}_{u}\left(x,u^{(0)}\left(x_{i-1}\right)\right)u^{(0)}\left(x_{i-1}\right).$$
The continuous solution of the Cauchy problem (\ref{NewDir_7}) can be represented in the following form
\begin{equation}\label{NewDir_1}
\begin{array}{c}
  u^{(j+1)}\left(x\right)=\left[\exp\left\{\int\limits_{x_{i-1}}^{x}q_{i}\left(\xi\right)d \xi\right\}+\right. \\[1.2em]
  +\int\limits_{x_{i-1}}^{x}\exp\left\{\int\limits_{\xi}^{x}q_{i}\left(\tau\right)d \tau\right\}N^{\prime}_{u}\left(\xi,u^{(0)}\left(x_{i-1}\right)\right)\times \\[1.2em]
  \left.\times \int\limits_{x_{i-1}}^{\xi}\frac{d}{d\eta}u^{(0)}\left(\eta\right)d\eta d\xi\right]u^{(j+1)}\left(x_{i-1}\right)+ \\[1.2em]
  +u^{(0)}\left(x_{i-1}\right)\int\limits_{x_{i-1}}^{x}\exp\left\{\int\limits_{\xi}^{x}q_{i}\left(\tau\right)d \tau\right\}N^{\prime}_{u}\left(\xi,u^{(0)}\left(x_{i-1}\right)\right)\int\limits_{x_{i-1}}^{\xi}\frac{d}{d \eta}u^{(j+1)}\left(\eta\right)d\eta d\xi+\\[1.2em]
  +\int\limits_{x_{i-1}}^{x}\exp\left\{\int\limits_{\xi}^{x}q_{i}\left(\tau\right)d \tau\right\}F^{(j+1)}\left(\xi\right)d \xi,\;x\in\left[x_{i-1},\;x_{i}\right].
\end{array}
\end{equation}
On the other hand, from \eqref{NewDir_11} we will obtain
\begin{equation}\label{Correction_1}
\begin{array}{c}
  u^{(j+1)}\left(x\right)=\left[\exp\left\{\int\limits_{x_{i-1}}^{x}n_{i}\left(\xi\right)d\xi\right\}+\right.\\[1.2em]
  \left.+\int\limits_{x_{i-1}}^{x}\exp\left\{\int\limits_{\xi}^{x}n_{i}\left(\eta\right)d\eta\right\} N^{\prime}_{u}\left(\xi, u^{(0)}\left(x_{i-1}\right)\right)u^{(0)}\left(\xi\right)d \xi\right]u^{(j+1)}\left(x_{i-1}\right)+ \\[1.2em]
  +\int\limits_{x_{i-1}}^{x}\exp\left\{\int\limits_{\xi}^{x}n_{i}\left(\eta\right)d \eta\right\}F^{(j+1)}\left(\xi\right)d\xi,
\end{array}
\end{equation}
where
$$n_{i}\left(x\right)=N\left(x, u^{(0)}\left(x_{i-1}\right)\right)\quad x\in\left[x_{i-1}, x_{i}\right].$$

If we differentiate (\ref{Correction_1}) with respect to $x$ we will obtain
$$$$
$$$$
\begin{equation}\label{NewDir_2}
\begin{array}{c}
  \frac{d}{d x}u^{(j+1)}\left(x\right)=\left[n_{i}\left(x\right)\exp\left\{\int\limits_{x_{i-1}}^{x}n_{i}\left(\xi\right)d \xi\right\}+N^{\prime}_{u}\left(x,u^{(0)}\left(x_{i-1}\right)\right)u^{(0)}\left(x\right)+\right. \\[1.2em]
  \left.+n_{i}\left(x\right)\int\limits_{x_{i-1}}^{x}\exp\left\{\int\limits_{\xi}^{x}n_{i}\left(\tau\right)d \tau\right\}N^{\prime}_{u}\left(\xi,u^{(0)}\left(x_{i-1}\right)\right)u^{(0)}\left(\xi\right) d\xi\right]u^{(j+1)}\left(x_{i-1}\right)+ \\[1.2em]
  + n_{i}\left(x\right)\int\limits_{x_{i-1}}^{x}\exp\left\{\int\limits_{\xi}^{x}n_{i}\left(\tau\right)d \tau\right\}F^{(j+1)}\left(\xi\right)d\xi+F^{(j+1)}\left(x\right),\;x\in\left[x_{i-1},\;x_{i}\right]. \\[1.2em]
\end{array}
\end{equation}
Equation \eqref{NewDir_2} could be expressed in the another form
\begin{equation}\label{1}
\begin{array}{c}
  \frac{d}{d x}u^{(j+1)}\left(x\right)=p_{i}\left(x\right)u^{(j+1)}\left(x_{i-1}\right)+ \\[1.2em]
  +n_{i}\left(x\right)\int\limits_{x_{i-1}}^{x}\exp\left\{\int\limits_{\xi}^{x}n_{i}\left(\tau\right)d
\tau\right\}F^{(j+1)}\left(\xi\right)d\xi+F^{(j+1)}\left(x\right),
\end{array}
\end{equation}
where
\begin{equation}\label{Correction_2}
    \begin{array}{c}
      p_{i}\left(x\right)=n_{i}\left(x\right)\exp\left\{\int\limits_{x_{i-1}}^{x}n_{i}\left(\xi\right)d \xi\right\}+N^{\prime}_{u}\left(x,u^{(0)}\left(x_{i-1}\right)\right)u^{(0)}\left(x\right)+ \\
      +n_{i}\left(x\right)\int\limits_{x_{i-1}}^{x}\exp\left\{\int\limits_{\xi}^{x}n_{i}\left(\tau\right)d \tau\right\}N^{\prime}_{u}\left(\xi,u^{(0)}\left(x_{i-1}\right)\right)u^{(0)}\left(\xi\right) d\xi,\;x\in \left[x_{i-1},\;x_{i}\right].
    \end{array}
\end{equation}
From \eqref{1}, \eqref{Correction_2} and lemma \ref{lema_pro_obmezh}  we obtain the following estimates
\begin{equation}\label{NewDir_3}
\begin{array}{c}
  \left\|\frac{d}{d x}u^{(j+1)}\left(x\right)\right\|_{0, \left[x_{i-1},\;x_{i}\right]}\leq \left\|p_{i}\left(x\right)\right\|_{0, \left[x_{i-1},\;x_{i}\right]}\left|u^{(j+1)}\left(x_{i-1}\right)\right|+ \\[1.2em]
  +\left(1+Nhe^{Nh}\right)\left\|F^{(j+1)}\left(x\right)\right\|_{0, \left[x_{i-1},\; x_{i}\right]},\quad N=\max\limits_{\substack{\left|u\right|\leq \mu \\ x\in\left[x_{0}, +\infty\right)}}\left|N\left(x, u\right)\right|
\end{array}
\end{equation}
for $ i=1,2,\ldots.$
For abbreviation, we denote  (using the result of lemma \ref{lema_pro_obmezh})
$$B=\max\limits_{\substack{\left|u\right|\leq \mu \\ x\in\left[x_{0}, +\infty\right) }}\left|N^{\prime}_{u}\left(x,u\right)\right|\left(\left|N\left(x, u\right)\right|\mu+k\right)<+\infty,$$
$$C=\max\limits_{\substack{\left|u\right|\leq \mu\\ x\in\left[x_{0}, +\infty\right) }}\left|N^{\prime}_{u}\left(x,u\right)\right|\mu<+\infty.$$
Then, from (\ref{NewDir_1}) we obtain the estimates
$$\left\|u^{(j+1)}\left(x\right)\right\|_{0,\left[x_{i-1},\;x_{i}\right]}\leq\left[1+\frac{B h_{i}^{2}}{2}\right]\left|u^{(j+1)}\left(x_{i-1}\right)\right|+$$
$$+\frac{C h_{i}^{2}}{2}\left\|\frac{d}{d x}u^{(j+1)}\left(x\right)\right\|_{0,\left[x_{i-1},\; x_{i}\right]}+h_{i}\left\|F^{(j+1)}\left(x\right)\right\|_{0,\left[x_{0}, \;+\infty\right)},$$

$$\left|u^{(j+1)}\left(x_{i}\right)\right|\leq\left[e^{-\alpha h_{i}}+\frac{B h_{i}^{2}}{2}\right]u^{(j+1)}\left(x_{i-1}\right)+$$
$$+\frac{Ch_{i}^{2}}{2}\left\|\frac{d}{d x}u^{(j+1)}\left(x\right)\right\|_{0,\left[x_{i-1},\; x_{i}\right]}+h_{i}\left\|F^{(j+1)}\left(x\right)\right\|_{0, \left[x_{0}, \;+\infty\right)},\; i=1,2,\ldots.$$
Combining the last two inequalities, the result of lemma \ref{lema_pro_obmezh}, and  (\ref{NewDir_3}), we get
\begin{equation}\label{NewDir_8}
    \left\|u^{(j+1)}\left(x_{i}\right)\right\|_{0,\left[x_{i-1}, \; x_{i}\right]}\leq\left[1+\overline{B} h_{i}^{2}\right]\left|u^{(j+1)}\left(x_{i-1}\right)\right|+h_{i}\overline{D}\left\|F^{(j+1)}\left(x\right)\right\|_{0, \left[x_{0}, \;+\infty\right)},
\end{equation}
\begin{equation}\label{NewDir_4}
    \left|u^{(j+1)}\left(x_{i}\right)\right|\leq\left[e^{-\alpha h_{i}}+\overline{B} h_{i}^{2}\right]\left|u^{(j+1)}\left(x_{i-1}\right)\right|+h_{i}\overline{D}\left\|F^{(j+1)}\left(x\right)\right\|_{0, \left[x_{0}, \;+\infty\right)},
\end{equation}
were
$$\overline{B}=\frac{B}{2}+\frac{C \overline{p}}{2},\;
\overline{D}=1+h_{i}\frac{C}{2}\left(1+Nhe^{Nh}\right),$$
$$\overline{p}= Ne^{Nh}+C+Nhe^{Nh}C=C+Ne^{Nh}\left(1+hC\right)\geq \left\|p\left(x\right)\right\|_{0, \left[x_{0}, +\infty\right)}.$$

Without loss of generality, we can assume that $\left\|F^{j+1}\left(x\right)\right\|_{0, \left[x_{0}, +\infty\right)}>0.$ Since that we denote
$$y_{i}=\left|u^{(j+1)}\left(x_{i}\right)\right|\left\|F^{(j+1)}\left(x\right)\right\|_{0,\left[x_{0}, \;+\infty\right)}^{-1}$$
and rewrite \eqref{NewDir_4} in the following form
\begin{equation}\label{NewDir_Poslid}
    y_{i}\leq \left[e^{-\alpha h_{i}}+\overline{B} h_{i}^{2}\right]y_{i-1}+h_{i}\overline{D}, y_{0}=0, i=1,2,\ldots.
\end{equation}
It is easy to see that for all $h_{i}>0$
$$e^{-\alpha h_{i}}+\overline{B} h_{i}^{2}\leq 1-\alpha h_{i}+h_{i}^{2}\left(\overline{B}+\frac{\alpha^{2}}{2}\right)= 1-h_{i}\left(\alpha - h_{i}\left(\overline{B}+\frac{\alpha^{2}}{2}\right)\right).$$
Thus the inequality
\begin{equation}\label{NewDir_5}
    h_{i}\leq h\leq \frac{\alpha}{2\overline{B}+\alpha^{2}},\; i=1,2,\ldots,
\end{equation}
implies
$$e^{-\alpha h_{i}}+\overline{B} h_{i}^{2}\leq 1-\frac{h_{i}\alpha}{2}.$$
The last inequality, under conditions
(\ref{NewDir_5}), guarantee that the recurrence sequence
\begin{equation}\label{NewDir_6}
  Y_{i}=\left(1-h_{i}\frac{\alpha}{2}\right)Y_{i-1}+h_{i}\overline{D},\; i=1,2,\ldots,\; Y_{0}=0,
\end{equation}
is a majorant for the sequence $y_{i}$ (\ref{NewDir_Poslid}).

It is easy to seen that \eqref{NewDir_6} can be reduced to the form
$$Z_{i}=\left(1-\overline{h}_{i}\right)Z_{i-1}, \; i=1,2,\ldots,\; Z_{0}=1,$$
where
$$\overline{h}_{i}=h_{i}\frac{\alpha}{2},\; Z_{i}=1-\overline{Y}_{i}\; \overline{Y}_{i}=\frac{\alpha Y_{i}}{2 \overline{D}},\; i=1,2,\ldots,$$
hence that
$$Z_{i}=\prod_{p=1}^{i}\left(1-\overline{h}_{p}\right),\; i=1,2,\ldots,\; Z_{0}=1.$$
We conclude from the last expression that the inequalities
\begin{equation}\label{NewDir_9}
 0<h_{i}\leq\mu_{1}=\min\left\{\frac{4}{\alpha},\;\frac{\alpha}{2\overline{B}+\alpha^{2}}\right\},\; \forall i\in\N
\end{equation}
implies the estimates
$$y_{i}\leq\overline{Y}_{i}\leq \frac{4}{\alpha}\overline{D}, \; \forall i\in \N.$$
From now on we orient the grid $\widehat{\omega}$ \eqref{Sitka} by requirement that \eqref{NewDir_9} is satisfied.
Thus the inequalities (\ref{NewDir_3}) and  (\ref{NewDir_8}) yield
$$\left\|u^{(j+1)}\left(x\right)\right\|_{0,\left[x_{i-1}, \; x_{i}\right]}\leq \left(\left(1+\overline{B}\mu_{1}^{2}\right)\frac{2}{\alpha}\overline{D}+\mu_{1}\overline{D}\right)\left\|F^{(j+1)}\left(x\right)\right\|_{0,\left[x_{0},+\infty\right)},$$
$$\left\|\frac{d}{d x}u^{(j+1)}\left(x\right)\right\|_{0,\left[x_{i-1}, \; x_{i}\right]}\leq \left(\overline{p}\overline{D}\frac{2}{\alpha}+\frac{Q}{\alpha}+1\right)\exp\left(\frac{C}{\alpha}\right)\left\|F^{(j+1)}\left(x\right)\right\|_{0,\left[x_{0},+\infty\right)}.$$
Combining the last two inequalities, we obtain the estimate
\begin{equation}\label{NewDir_10}
   \left\|u^{(j+1)}\left(x\right)\right\|_{1,\left[x_{0}, +\infty\right)}\leq \sigma \left\|F^{(j+1)}\left(x\right)\right\|_{0,\left[x_{0}, +\infty\right)},
\end{equation}
where
$$\sigma=\max\left\{\left(\left(1+\overline{B}\mu_{1}^{2}\right)\frac{2}{\alpha}\overline{D}+\mu_{1}\overline{D}\right),\;\left(\overline{p}\overline{D}\frac{2}{\alpha}+\frac{Q}{\alpha}+1\right)\exp\left(\frac{C}{\alpha}\right)\right\}.$$

To make the following estimations we need to use the results of lemmas \ref{lema_1} and \ref{lema_2}. The requirements of these lemmas are satisfied if we set (in the notations of lemma \ref{lema_1})
\begin{equation}\label{formula_for_n_tilde}
\widetilde{N}\left(u\right)=\sum_{i=0}^{\infty}\left\|a_{i}\left(x\right)\right\|_{0}u^{i}\leq \sum_{i=0}^{\infty}B_{i}u^{i},\; \forall u\in\R.
\end{equation}

(\ref{NewDir_10}) is equivalent to
$$\left\|u^{(j+1)}\left(x\right)\right\|_{1}\leq $$
$$\leq\sigma\left\{\sum_{p=1}^{j}A_{j+1-p}\left(\widetilde{N}\left(u\right);\; \left\|u^{(0)}\left(x\right)\right\|_{1}, \ldots, \left\|u^{(j+1-p)}\left(x\right)\right\|_{1}\right)
\left\|u^{(p)}\left(x\right)\right\|_{1}\right.+$$
$$+h\sum_{p=0}^{j}A_{j-p}\left(\widetilde{N}'\left(u\right)u;\;\left\|u^{(0)}\left(x\right)\right\|_{1},\ldots, \left\|u^{(j-p)}\left(x\right)\right\|_{1}\right)\left\|u^{(p)}\left(x\right)\right\|_{1}+$$
\begin{equation}\label{NewDir_15}
    +\frac{\left\|u^{(0)}\left(x\right)\right\|_{1}}{\left(j+1\right)!}\left[\frac{d^{j+1}}{d z^{j+1}}\left(\widetilde{N}\left(\sum_{s=0}^{\infty}z^{s}\left\|u^{(s)}\left(x\right)\right\|_{1}\right)\right.-\right.
\end{equation}
$$\left.\left.\left.-\sum_{s=1}^{\infty}z^{s}\left\|u^{(s)}\left(x\right)\right\|_{1}\widetilde{N}'\left(\left\|u^{(0)}\left(x\right)\right\|_{1}\right)\right)\right]_{z=0}\right\},\; j=0,1,2,\ldots.$$
Denote
\begin{equation}\label{NewDir_16}
   \nu_{j}=\frac{\left\|u^{\left(j\right)}\left(x\right)\right\|_{1}}{h^{j}},
   j=0,1,\ldots.
\end{equation}
Let us  define a sequence $\left\{V_{i}\right\}_{i=1}^{\infty}$ by the following recursive formula
   \begin{equation}\label{NewDir_17}
     V_{j+1}=\sigma\left\{\sum_{p=1}^{j}A_{j+1-p}\left(\widetilde{N}\left(u\right);V_{0},\ldots,V_{j+1-p}\right)V_{p}+\right.
   \end{equation}
   $$+\sum_{p=0}^{j}A_{j-p}\left(\widetilde{N}'\left(u\right)u;V_{0},\ldots,V_{j-p}\right)V_{p}+$$
   $$\left.+\frac{V_{0}}{\left(j+1\right)!}\frac{d^{j+1}}{d z^{j+1}}\left(\widetilde{N}\left(\sum_{s=0}^{\infty}z^{s}V_{s}\right)\right)_{z=0}-V_{j+1}V_{0}\widetilde{N}'\left(V_{0}\right)\right\},\;j=0,1,\ldots.$$
   or, in more convenient form,
   \begin{equation}\label{NewDir_18}
   V_{j+1}=\frac{\sigma}{1+\sigma V_{0}\widetilde{N}'\left(V_{0}\right)}\left\{\sum_{p=0}^{j}A_{j+1-p}\left(\widetilde{N}\left(u\right); V_{0},\ldots, V_{j+1-p}\right)V_{p}+\right.
   \end{equation}
   $$\left.+\sum_{p=0}^{j}A_{j-p}\left(\widetilde{N}'\left(u\right)u;V_{0},\ldots, V_{j-p}\right)V_{p}\right\},$$
   where $V_{0}=\mu.$

    It is easy to seen that $\nu_{i}\leq V_{i},$ $\forall i\in \N\bigcup\left\{0\right\}.$

    We are now in a position to prove that for $h$ \eqref{Sitka} sufficiently small the assumptions (\ref{OFD_5}) and (\ref{OFD_5'}) hold. In order to do that, it is sufficient to show that the function \begin{equation}\label{NewDir_19_1}
                        g\left(z\right)=\sum_{j=0}^{\infty}z^{j}V_{j}
    \end{equation} has  a nonempty open domain.

                  From (\ref{NewDir_18}) we obtain
                    \begin{equation}\label{NewDir_19}
                        g\left(z\right)-V_{0}=\frac{\sigma}{1+\sigma V_{0}\widetilde{N}'\left(V_{0}\right)}\left\{g\left(z\right)\left[\widetilde{N}\left(g\left(z\right)\right)-\widetilde{N}\left(V_{0}\right)\right]+zg^{2}\left(z\right)\widetilde{N}'\left(g\left(z\right)\right)\right\}.
                    \end{equation}
                    Let us express $z$ from (\ref{NewDir_19}) as a function of $g:$
                           $$ z\left(g\right)=\frac{1}{g^{2}
                           \widetilde{N}'\left(g\right)}\left\{\frac{1}{\Sigma}\left(g-V_{0}\right)-\left(\widetilde{N}\left(g\right)-\widetilde{N}\left(V_{0}\right)\right)g\right\},$$
                    \begin{equation}\label{NewDir_20}
                            V_{0}\leq g,\; \Sigma=\frac{\sigma}{1+\sigma V_{0}\widetilde{N}'\left(V_{0}\right)}.
                    \end{equation}
                           Evidently, the function $z\left(g\right)$ (\ref{NewDir_20}) is defined and continuously differentiable in some open neighbourhood of the point $g=V_{0}.$ To prove the existence of the inverse function $g=g\left(z\right)$ defined in some open neighbourhood of the point $z=0$ it is sufficient to show that $z'\left(V_{0}\right)>0.$ The last fact immediately follows from formula (\ref{NewDir_20}):
                    $$z'\left(V_{0}\right)=\lim_{g \rightarrow V_{0}}\frac{z\left(g\right)-z\left(V_{0}\right)}{g-V_{0}}=$$
                    \begin{equation}\label{NewDir_21}
                        =\lim_{g \rightarrow V_{0}}\frac{1}{V_{0}^{2}\widetilde{N}'\left(V_{0}\right)}\left(\frac{1}{\Sigma}-\frac{\widetilde{N}\left(g\right)-\widetilde{N}\left(V_{0}\right)}{g-V_{0}}g\right)=
                    \end{equation}
                    $$=\frac{1}{\sigma V_{0}^{2}\widetilde{N}'\left(V_{0}\right)}>0.$$
                    From \eqref{formula_for_n_tilde} it follows that  $\widetilde{N}\left(u\right)\rightarrow +\infty$ as $u\rightarrow +\infty.$
                    Hence, from equation (\ref{NewDir_20}) we have:
                    $$\lim_{g \rightarrow +\infty}z\left(g\right)\leq 0.$$
                     The last inequality together with (\ref{NewDir_21}) provides the existence of a point $g_{max}$ such that  $g_{max}>V_{0},\;z'\left(g_{max}\right)=0$ and $\forall g\in\left(V_{0},\;g_{max}\right)\; z'\left(g\right)>0.$  $R=z_{max}=z\left(g_{max}\right)$ is the radius of convergence of the power series (\ref{NewDir_19_1}). In other words, $$R^{j}V_{j}\leq C\frac{1}{\left(j+1\right)^{1+\varepsilon}},$$
                    for sufficiently small $\varepsilon>0,$ where $C$ is some constant independent of $h_{i},\; i\in\N.$

                    Returning to notation \eqref{NewDir_16}, we obtain
                    \begin{equation}\label{NewDir_22}
                        \left\|u^{(j)}\left(x\right)\right\|_{1,\left[x_{0},\;+\infty\right)}\leq \frac{C}{\left(j+1\right)^{1+\varepsilon}}\left(\frac{h}{R}\right)^{j}, j=0,1,\ldots.
                    \end{equation}
                    This estimate gives us the following sufficient condition for series (\ref{NewDir_19_1}) to be convergent:
                    $$\frac{h}{R}\leq 1.$$

                    Finally, we have proved that for $h>0$ sufficiently small (to be exact, $h<\min\left\{\mu_{1},\; R\right\}$)
                     assumptions (\ref{OFD_5}) and (\ref{OFD_5'}) are fulfilled. Therefore, it is remain to prove that the sum of the uniformly convergent series (\ref{ZagalniyRyad}) is a solution to problem (\ref{OFD_1}). To do that we need to add up the base problem (\ref{OFD_2}) and equations (\ref{NewDir_11}) for
                     $j=0,1,\ldots.$ As a result we obtain
                     \begin{equation}\label{NewDir_23}
                       \sum_{j=0}^{\infty}\frac{d}{d x}u^{(j)}\left(x\right)-\sum_{j=0}^{\infty}\sum_{p=0}^{j}A_{j-p}\left(N\left(x,\left(\cdot\right)\right);\;u^{(0)}\left(x\right),\ldots, u^{(j-p)}\left(x\right)\right)u^{(p)}\left(x\right)=\phi\left(x\right),
                     \end{equation}
$x\in\bigcup\limits_{i=1}^{\infty}\left(x_{i-1},\; x_{i}\right).$ Using the theorem about the composition of two power series (see \cite{Fihtenholts}, p. 485), it is easy to prove the equality
$$N\left(x,\sum_{i=0}^{\infty}t^{i}u^{(i)}\left(x\right)\right)\sum_{i=0}^{\infty}t^{i}u^{(i)}\left(x\right)=$$
 $$=\sum_{j=0}^{\infty}t^{j}\sum_{p=0}^{j}A_{j-p}\left(N\left(x,\left(\cdot\right)\right), u^{(0)}\left(x\right),\ldots, u^{(j-p)}\left(x\right)\right)u
^{(p)}\left(x\right),$$
 $\forall t\in \left[0,\;1\right],\;\forall
x\in \left[x_{0},\;+\infty\right).$ Thus, taking into account the uniformly convergence of the series $\sum\limits_{i=0}^{\infty}\frac{d}{d
x}u^{(i)}\left(x\right)$ {on each interval $\left(x_{i-1}, x_{i}\right),\; i=1,2,\ldots$}, we can rewrite (\ref{NewDir_23}) in the following form
\begin{equation}\label{NewDir_24}
    \frac{d}{d
    x}\tilde{u}\left(x\right)-N\left(\tilde{u}\left(x\right)\right)\tilde{u}\left(x\right)=\phi\left(x\right), \;
    \forall x\in\bigcup_{i=1}^{\infty}\left(x_{i-1},\; x_{i}\right),\;\tilde{u}\left(x\right)=\sum\limits_{j=0}^{\infty}u^{(j)}\left(x\right),
\end{equation}
Recall that $\tilde{u}\left(x_{0}\right)=u_{0}.$ Then the continuity of $\tilde{u}\left(x\right),$ the existence and uniqueness of the solution $u\left(x\right)$ of Cauchy problem $(\ref{OFD_1})$ on $\left[x_{0},\;+\infty\right)$ along with (\ref{NewDir_24}) imply our final goal
$\tilde{u}\left(x\right)\equiv u\left(x\right),\; \forall x\in\left[x_{0},\;+\infty\right).$ This means that the FD-method for the Cauchy problem \eqref{OFD_1} converges to the exact solution of the problem in the sense of definition \ref{defin}.

We leave it to the reader to verify the error estimates \eqref{error_estimations1}, \eqref{error_estimations2}. They could be easily derived from \eqref{NewDir_22}.
\end{proof}

\section{Examples.}
{\bf Example 1.} As an example let us consider the following Cauchy problem
\begin{equation}\label{Pr_1}
    \frac{d}{d x}u\left(x\right)=-u^{3}\left(x\right)-u\left(x\right)+\cos\left(x\right)+\sin\left(x\right)+\sin^{3}\left(x\right), \; u\left(0\right)=0.
\end{equation}
It easy to seen that the exact solution of problem  (\ref{Pr_1}) is
 $u^{\ast}\left(x\right)=\sin(x).$

To solve problem (\ref{Pr_1}) {numerically}, we, {first of all, may try to} apply the ADM. We approximate the solution by the  $m-$th partial sum
$\overset{m}{u_{A}}$ of the series
$\sum\limits_{i=0}^{\infty}u^{(i)}_{A}\left(x\right),$ where
$u^{(i)}_{A}\left(x\right)$ can be found from the {sequence} of Cauchy problems
$$\frac{d}{d x}u^{(0)}_{A}\left(x\right)=-u^{(0)}_{A}\left(x\right)+\cos\left(x\right)+\sin\left(x\right)+\sin^{3}\left(x\right),\; u^{(0)}_{A}\left(0\right)=0,$$
$$\frac{d}{d x}u^{(i)}_{A}\left(x\right)=-\left.\frac{d}{d t}t\left(\sum_{i=0}^{\infty}t^{i}u^{(i)}_{A}\left(x\right)\right)^{3}\right|_{t=0}-u^{(i)}_{A}\left(x\right),\;u^{(i)}_{A}\left(0\right)=0,\;i=1,2,\ldots.$$
The results are presented on the Fig. 1.
\begin{figure}[h!]
\begin{minipage}[h]{1\linewidth}
\begin{minipage}[h]{0.48\linewidth}
\center{\rotatebox{0}{\includegraphics[height=1\linewidth,
width=1.3\linewidth]{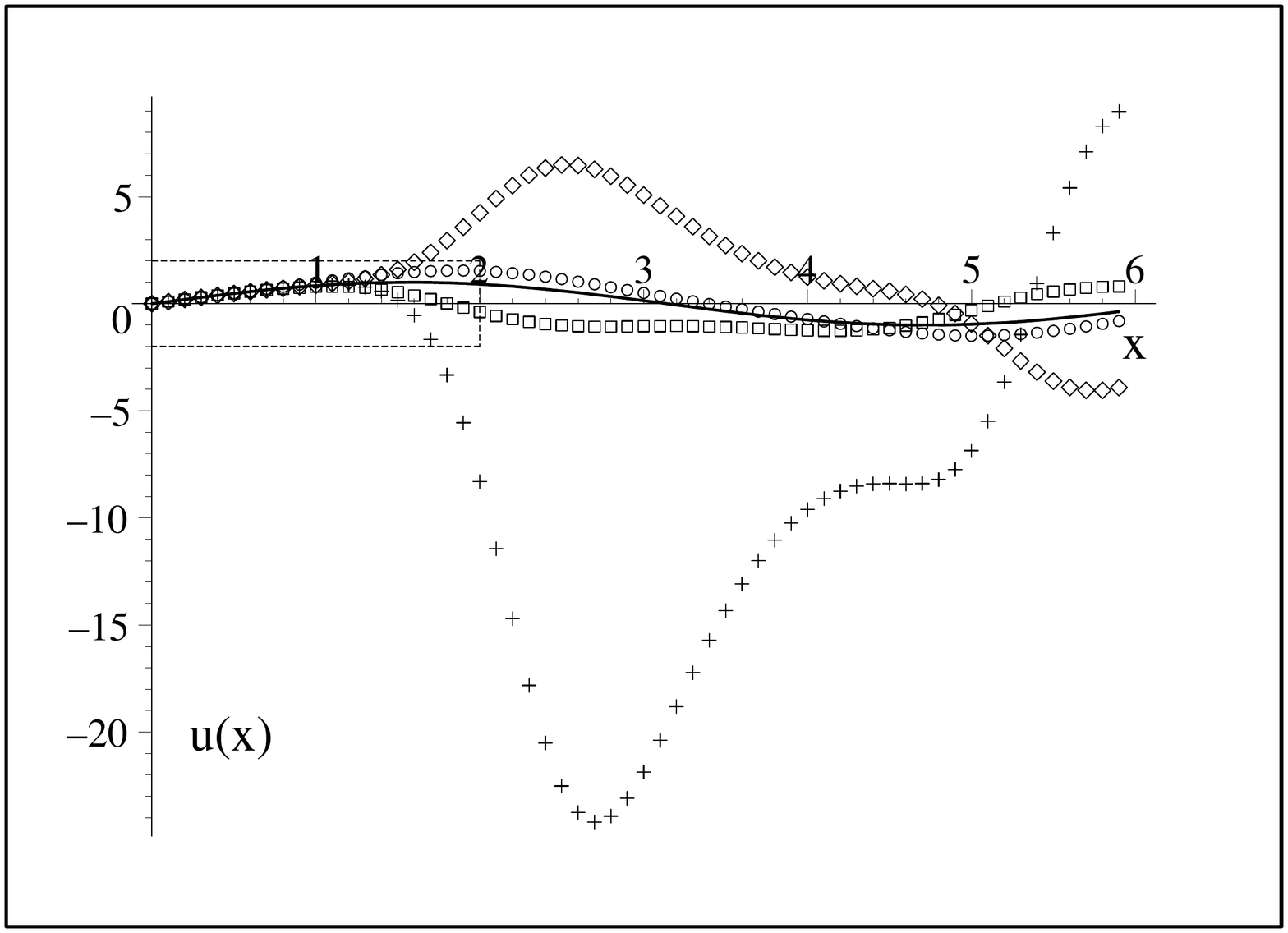}} \\ a) }
\end{minipage}
\hfill\label{image1}
\begin{minipage}[h]{0.48\linewidth}
\center{\rotatebox{0}{\includegraphics[height=1\linewidth,
width=1.3\linewidth]{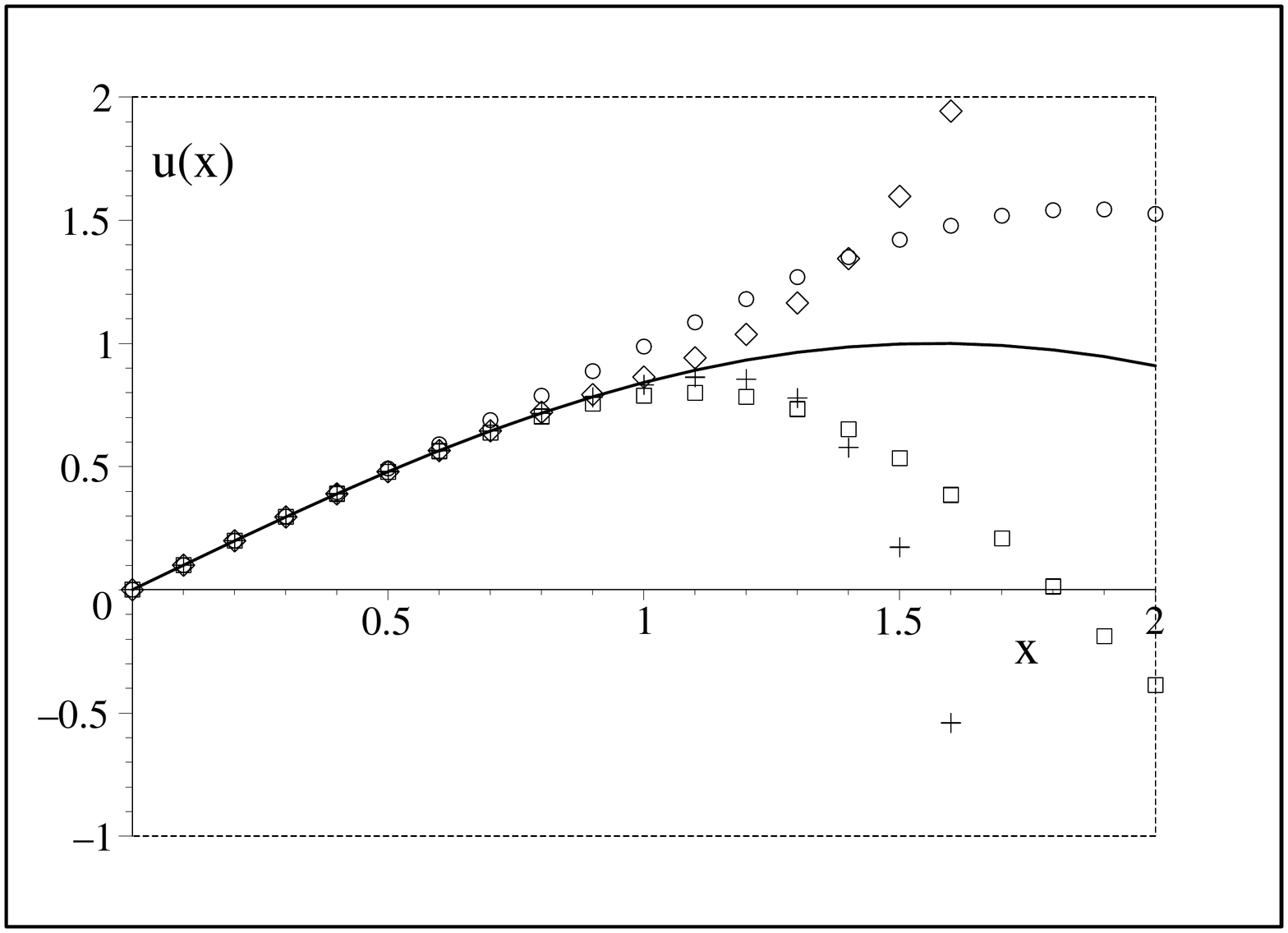}} b)}
\end{minipage}
\caption{Example 1. ADM, {\it continuous line:}
$\;u^{\ast}\left(x\right)=\sin\left(x\right);$ $\;\circ:\;
\overset{0}{u}_{A}\left(x\right);$ $\;\Box:\;
\overset{1}{u}_{A}\left(x\right);$ $\;\lozenge:\;
\overset{2}{u}_{A}\left(x\right);\;$ $\;+:\;
\overset{3}{u}_{A}\left(x\right);$}
\end{minipage}
\end{figure}
The graphs on Fig. 1 show that the ADM for the Cauchy problem (\ref{Pr_1}) is divergent on $\left[0,\; 6\right]$ .

The application of FD-method  to problem (\ref{Pr_1}) with the grid
$$\omega=\left\{x_{0}=0,\;x_{i}=\frac{1}{3}i,\; i=1\ldots,
144\right\}$$ yields results which are presented on Fig.2, Fig.3. Here we use {the} notations:
$$\delta_{i}\left(x\right)=\sin\left(x\right)-\overset{i}{u}\left(x\right),\;
i=0,1,\ldots.$$

\begin{figure}[h!]
\begin{minipage}[h]{1\linewidth}
\begin{minipage}[h]{0.48\linewidth}
\center{\rotatebox{-0}{\includegraphics[height=1\linewidth,
width=1.3\linewidth]{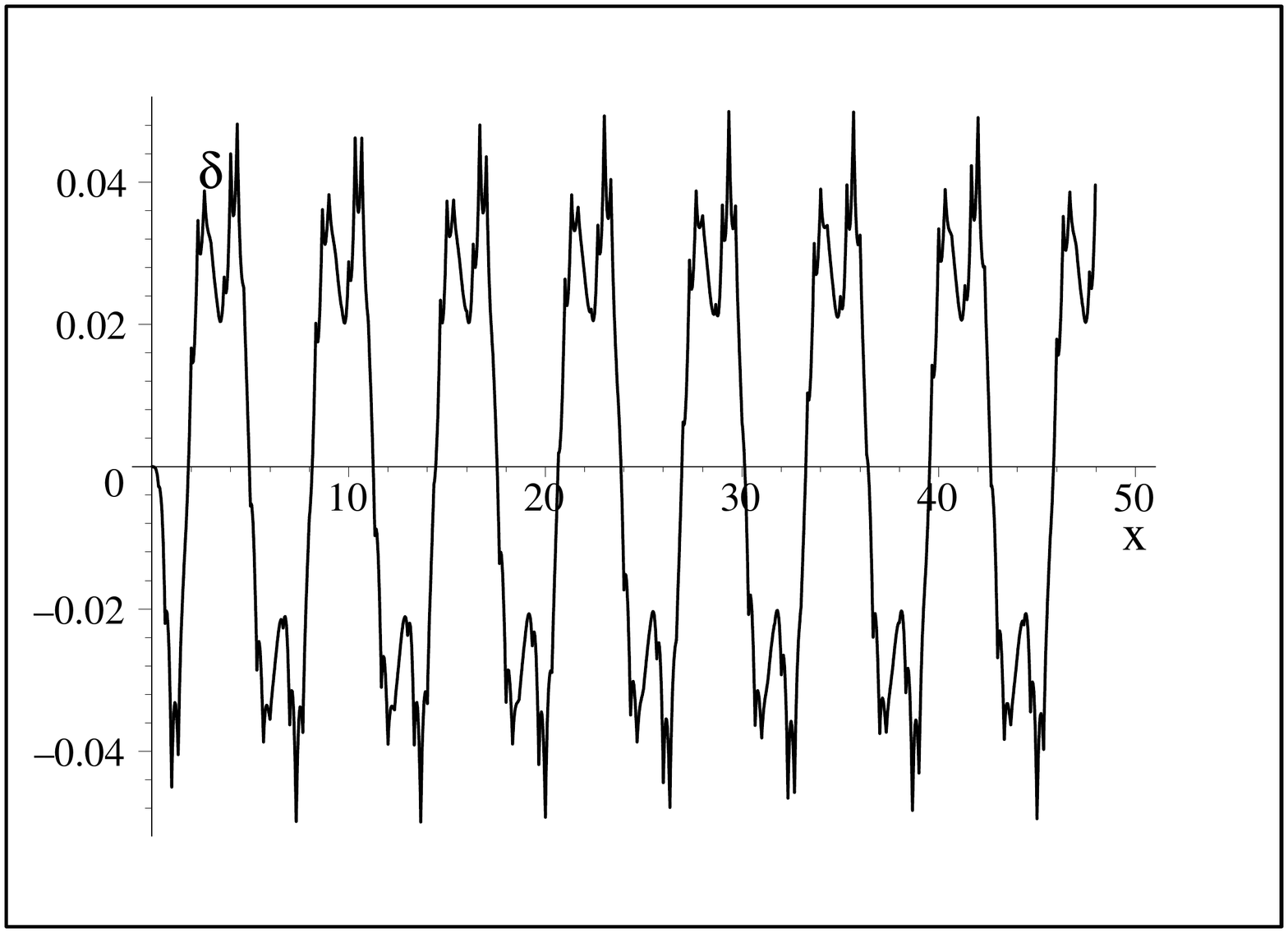}} \\ a) }
\end{minipage}
\hfill\label{image1}
\begin{minipage}[h]{0.48\linewidth}
\center{\rotatebox{-0}{\includegraphics[height=1\linewidth,
width=1.3\linewidth]{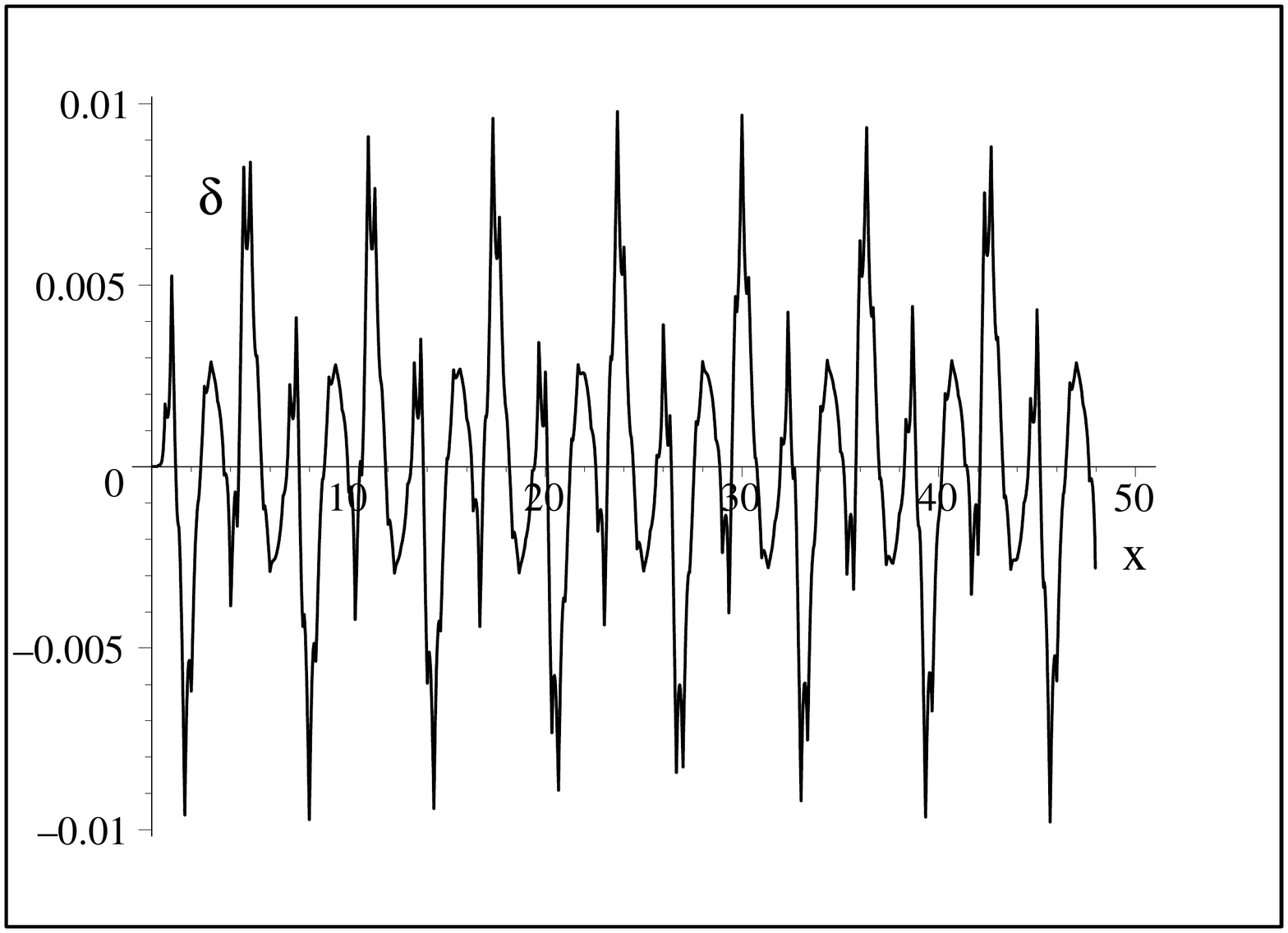}} \\ b)}
\end{minipage}
\caption{Example 1. FD-method. The graphs of absolute errors. a) --
$\;\delta_{0}\left(x\right);\;$ b) -- $\;\delta_{1}\left(x\right).$}\end{minipage}
\end{figure}

\begin{figure}[h!]
\begin{minipage}[h]{1\linewidth}
\begin{minipage}[h]{0.48\linewidth}
\center{\rotatebox{-0}{\includegraphics[height=1\linewidth,
width=1.3\linewidth]{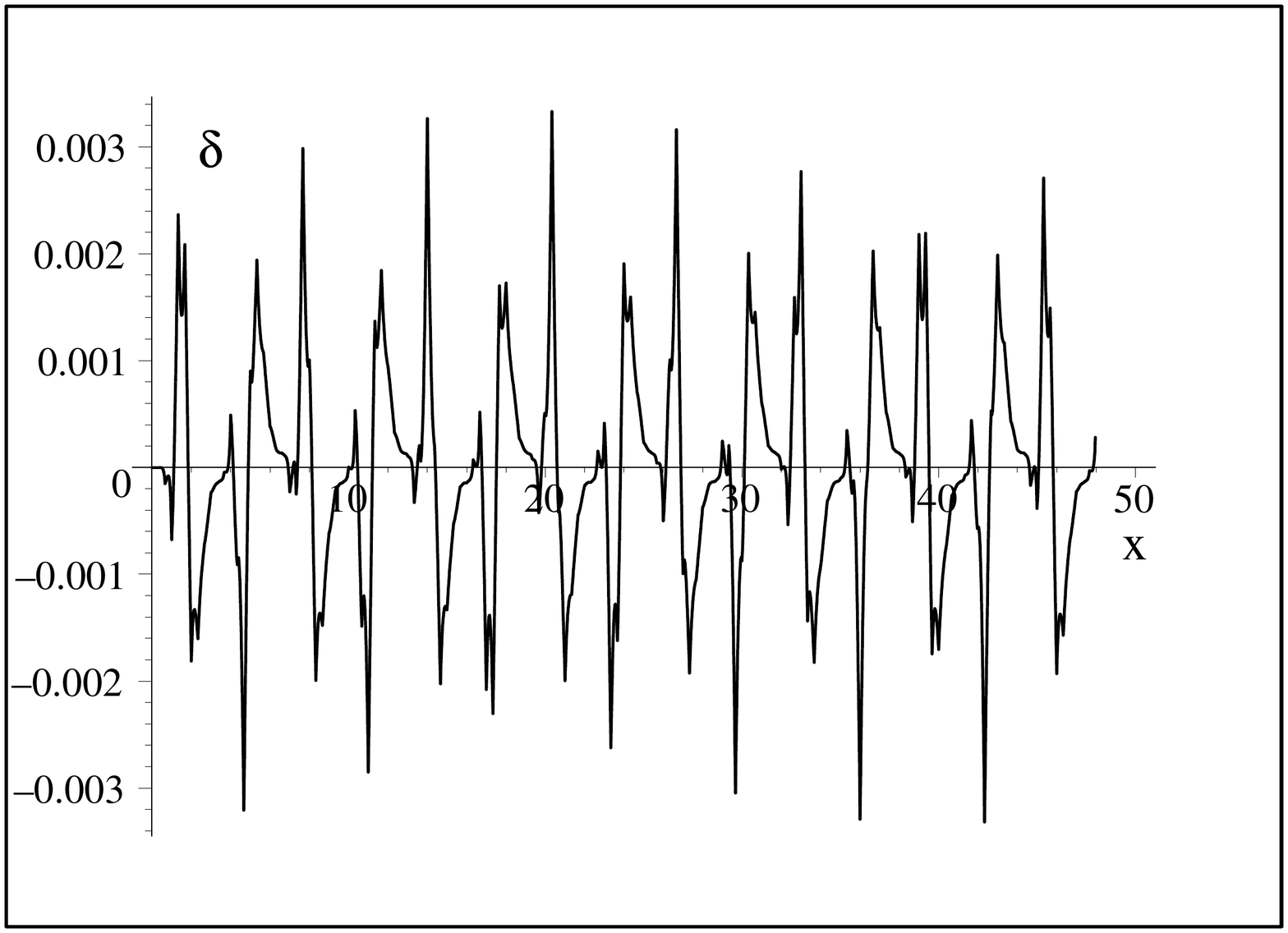}} \\ a)}
\end{minipage}
\hfill\label{image1}
\begin{minipage}[h]{0.48\linewidth}
\center{\rotatebox{-0}{\includegraphics[height=1\linewidth,
width=1.3\linewidth]{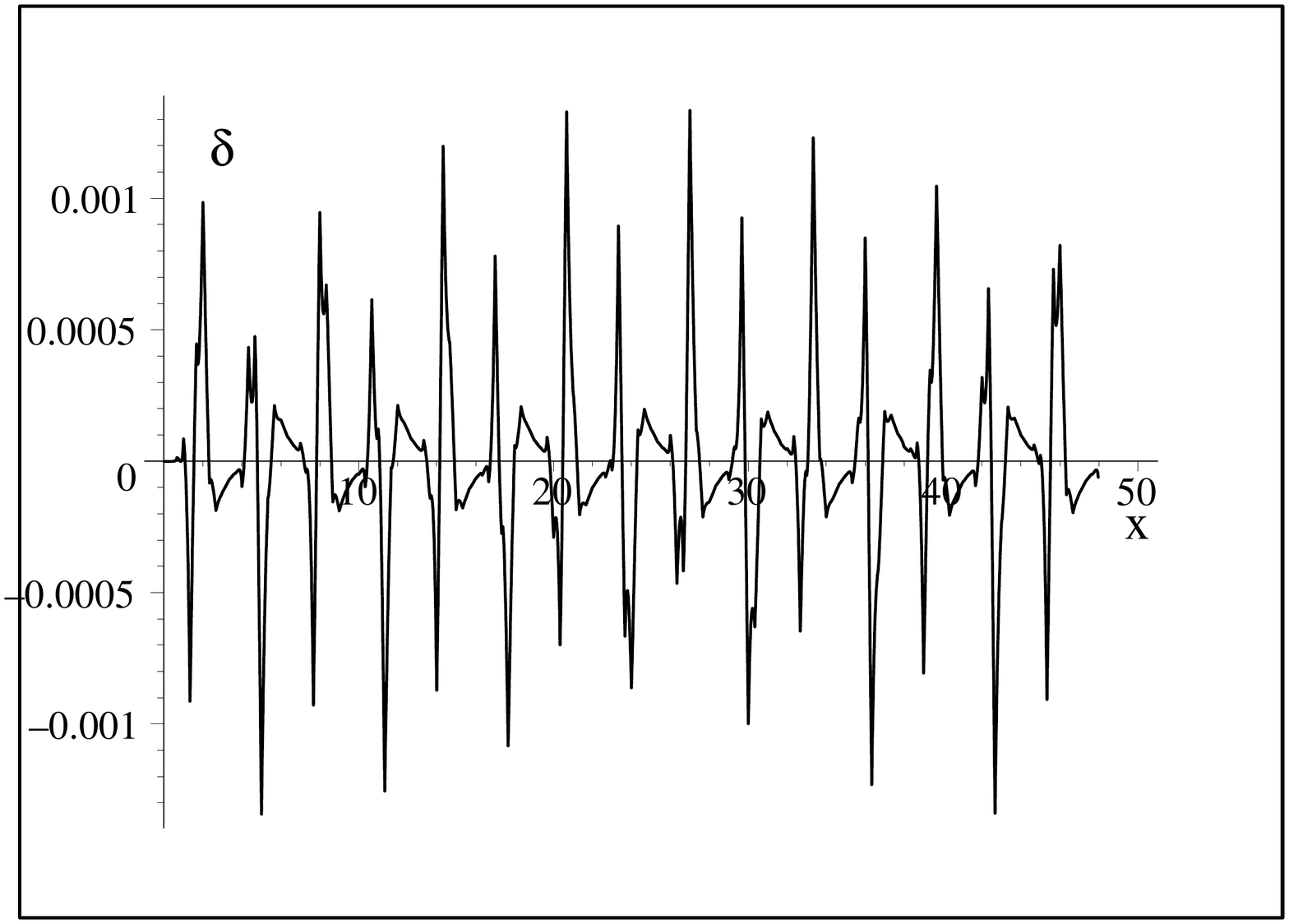}} \\ b)}
\end{minipage}
\caption{Example 1. FD-method. The graphs of absolute errors.
a) -- $\;\delta_{2}\left(x\right);\;$ b) --
$\;\delta_{3}\left(x\right);\;$$x\in\left[0,\;48\right].$}
\end{minipage}
\end{figure}

\begin{figure}[h!]
\begin{minipage}[h]{1\linewidth}
\begin{minipage}[h]{0.48\linewidth}
\center{\rotatebox{-0}{\includegraphics[height=1\linewidth,
width=1.3\linewidth]{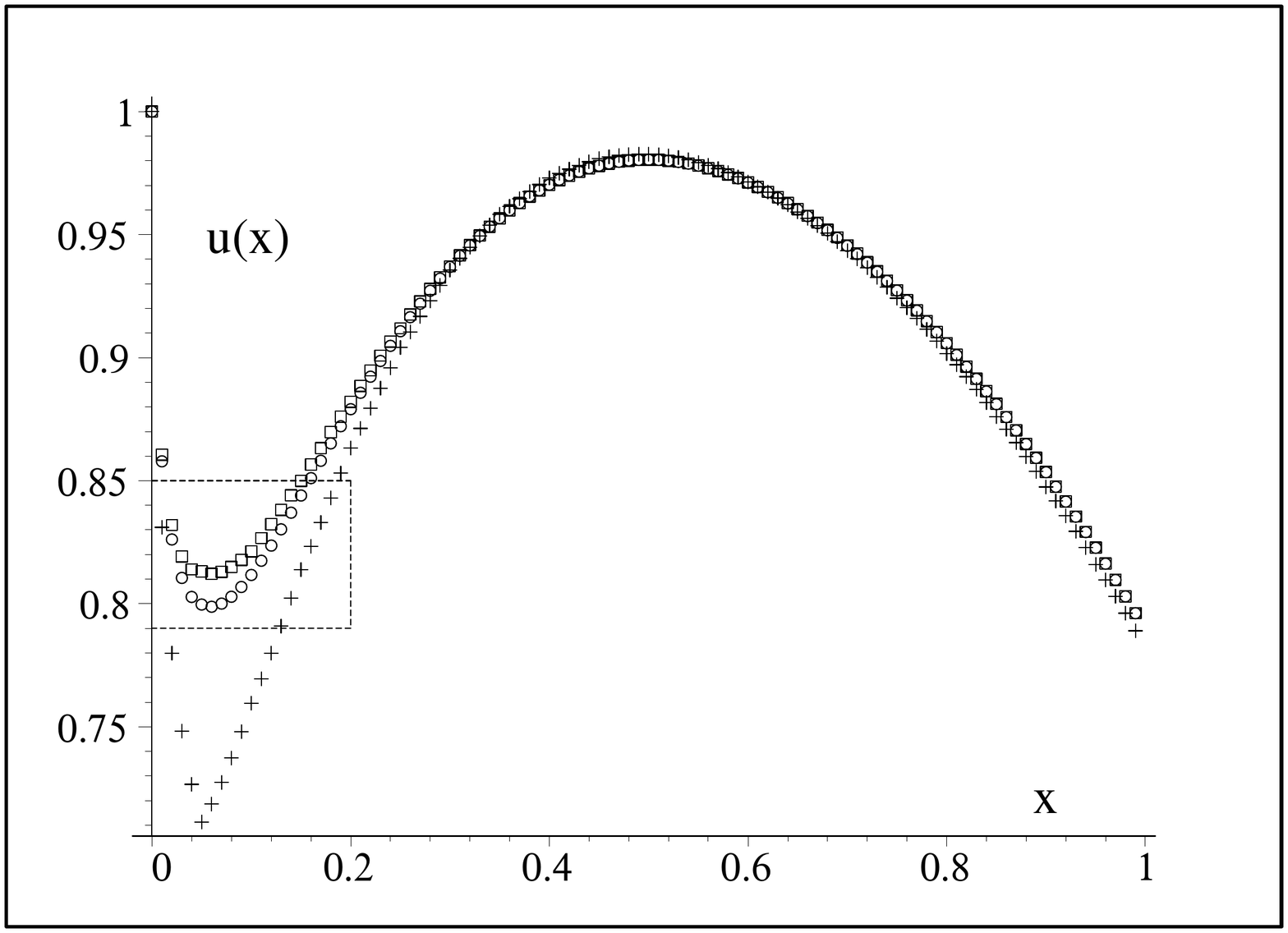}} \\  }
\end{minipage}
\hfill\label{image1}
\begin{minipage}[h]{0.48\linewidth}
\center{\rotatebox{-0}{\includegraphics[height=1\linewidth,
width=1.3\linewidth]{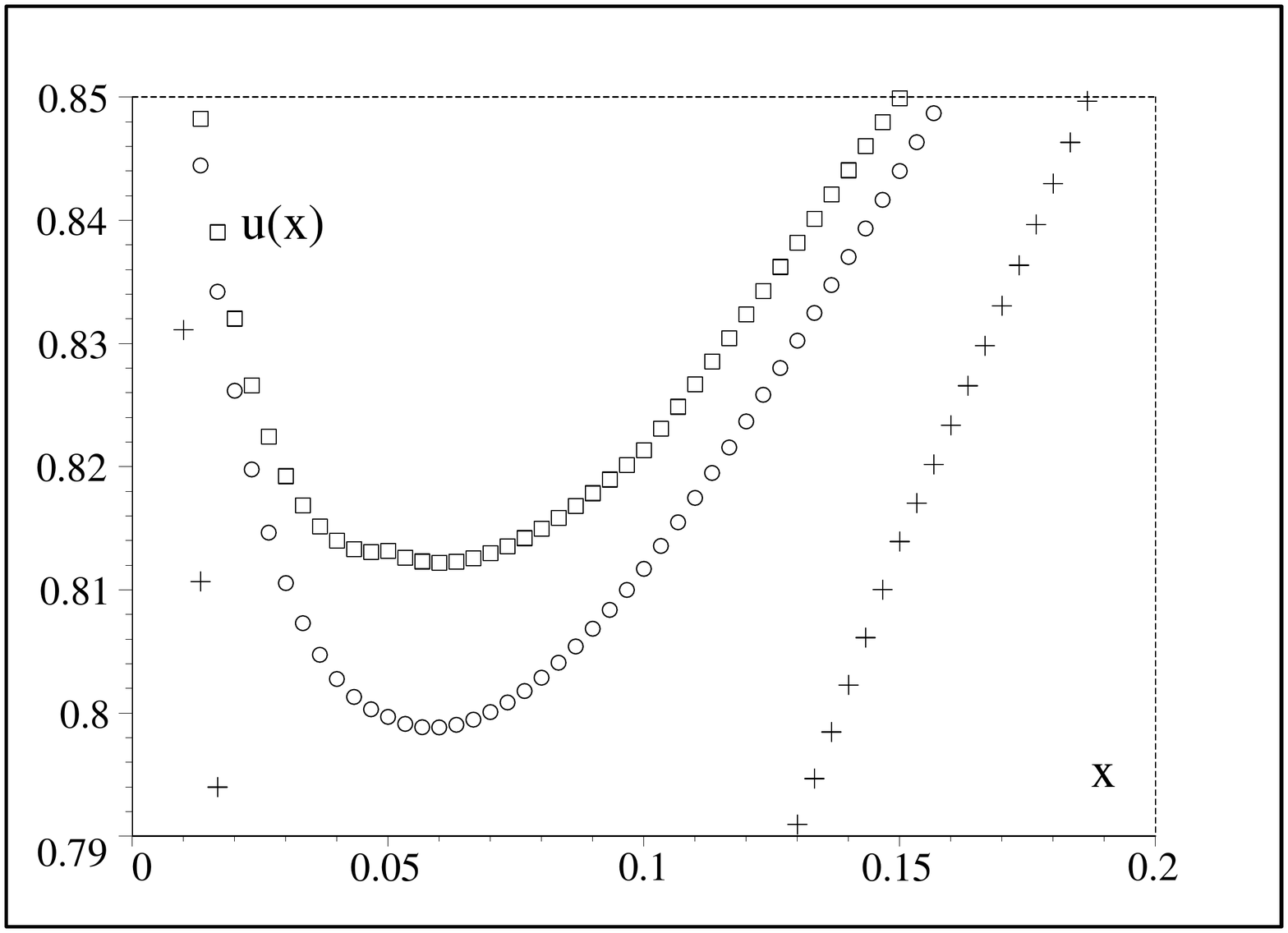}}}
\end{minipage}
\caption{Example 2. FD-method.
$+:\;\overset{0}{u}\left(x\right),\;\;\Box:\;
\overset{1}{u}\left(x\right),\;\;\circ:\;
\overset{2}{u}\left(x\right).$}
\end{minipage}
\end{figure}

\begin{figure}[h!]
\begin{minipage}[h]{1\linewidth}
\begin{minipage}[h]{0.48\linewidth}
\center{\rotatebox{-0}{\includegraphics[height=1\linewidth,
width=1.3\linewidth]{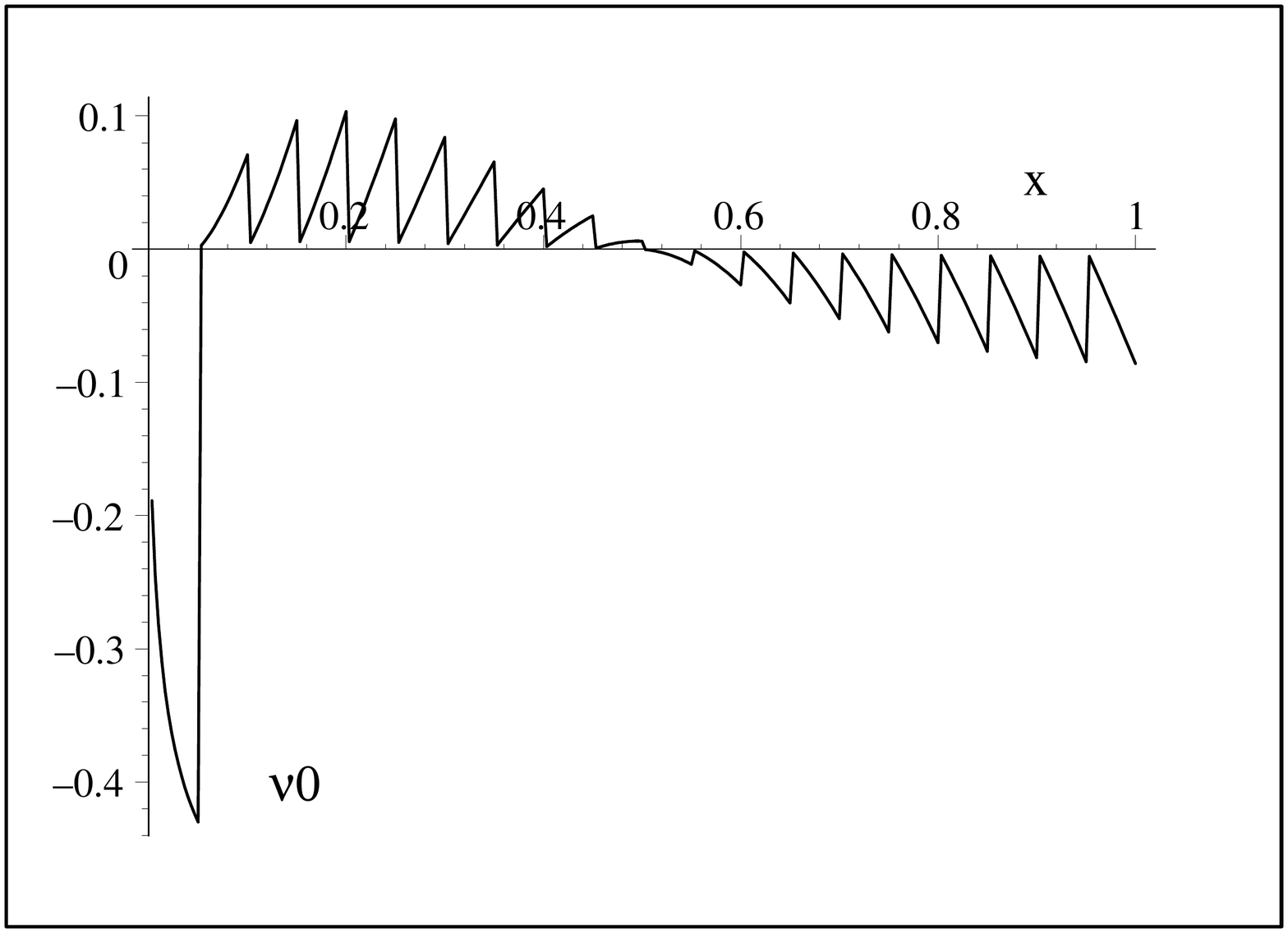}} \\ a)}
\end{minipage}
\hfill\label{image1}
\begin{minipage}[h]{0.48\linewidth}
\center{\rotatebox{-0}{\includegraphics[height=1\linewidth,
width=1.3\linewidth]{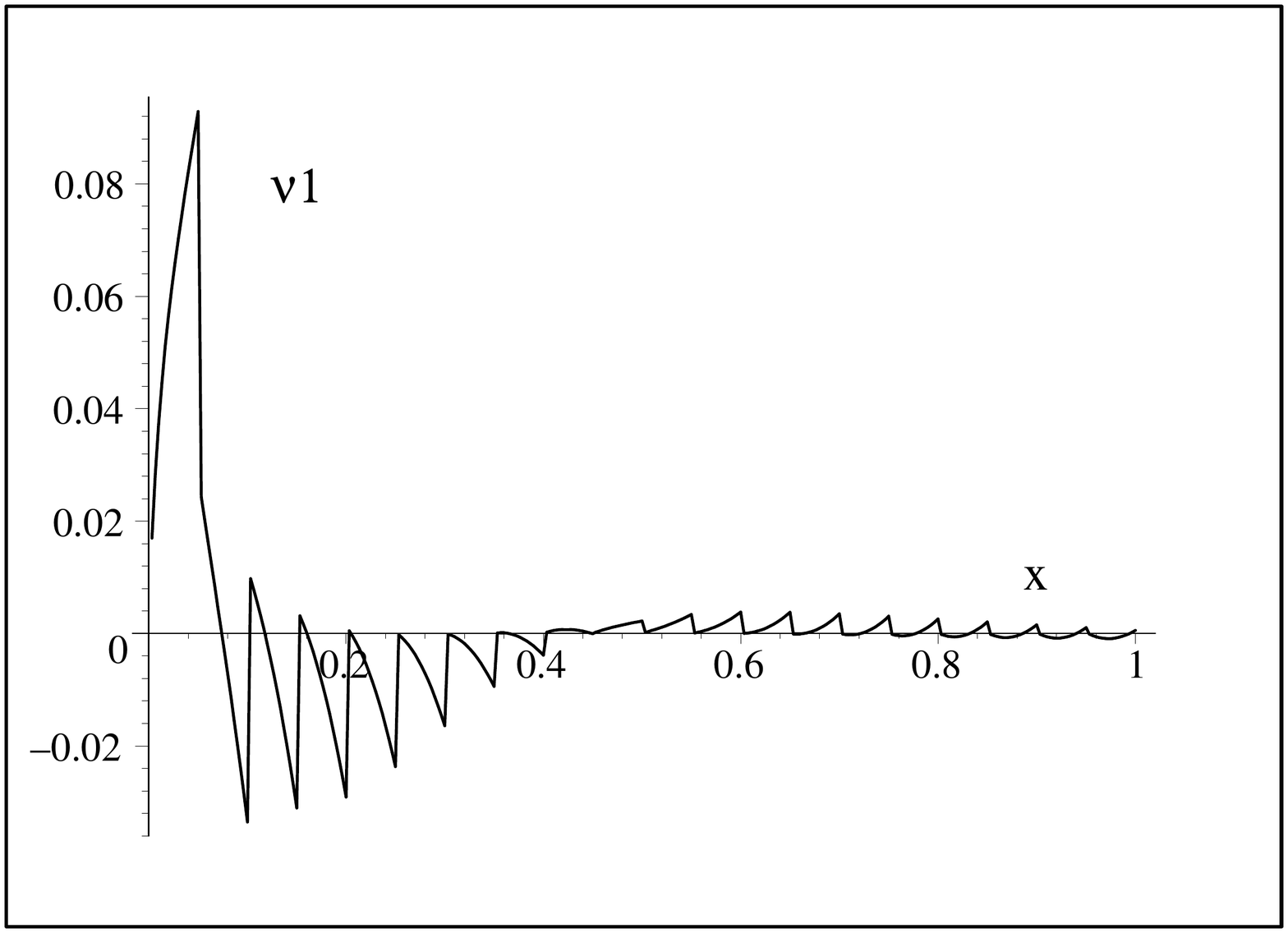}}\\ b)}
\end{minipage}
\caption{Example 2. FD-method. Graphs of discrepancy. \newline a) --
$\nu_{0}\left(x\right),\;\;$ b) -- $\nu_{1}\left(x\right).$}
\end{minipage}
\end{figure}

\begin{figure}[h!]
\begin{minipage}[h]{1\linewidth}
\begin{minipage}[h]{0.48\linewidth}
\center{\rotatebox{-0}{\includegraphics[height=1\linewidth,
width=1.3\linewidth]{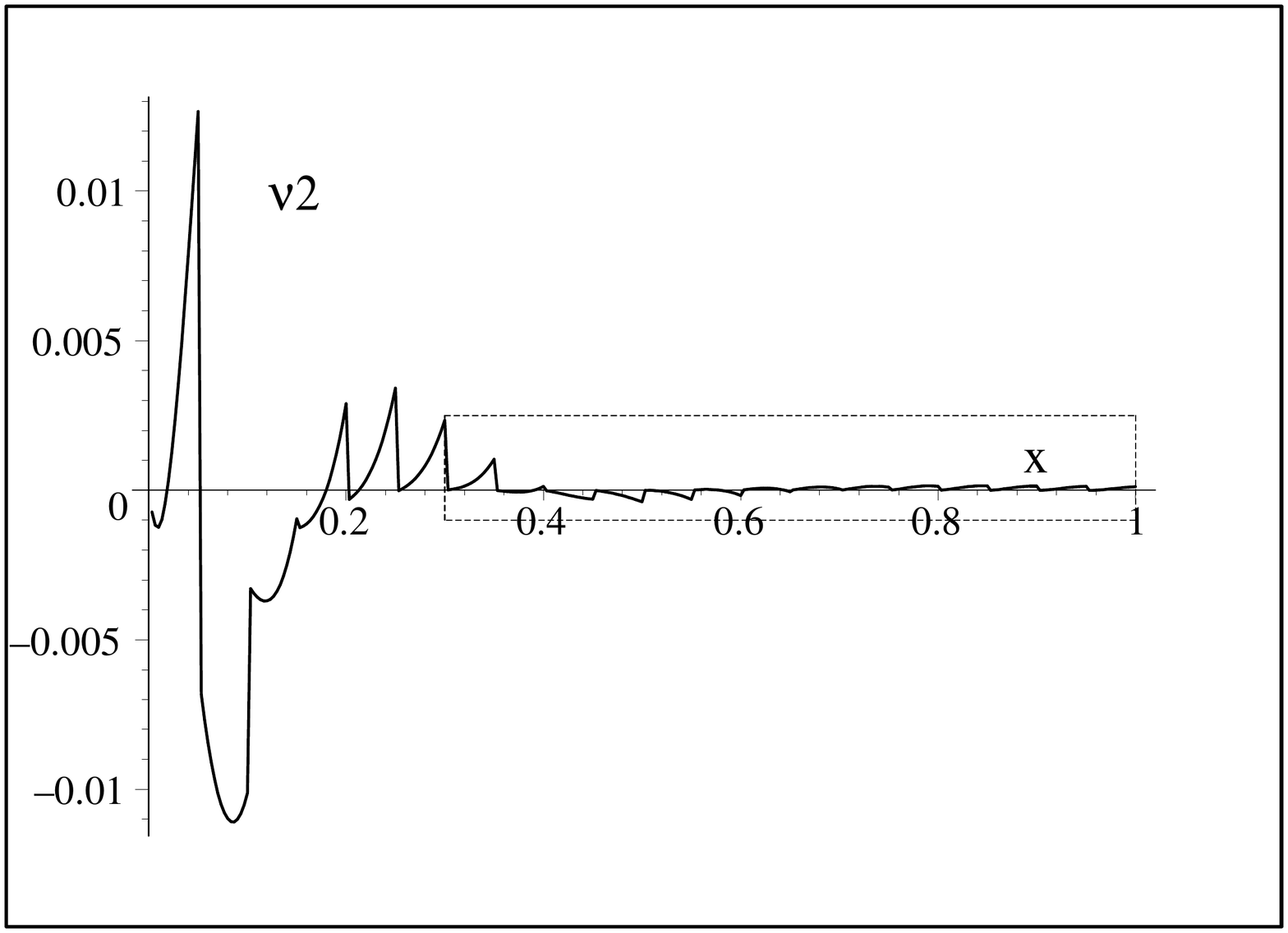}} \\  }
\end{minipage}
\hfill\label{image1}
\begin{minipage}[h]{0.48\linewidth}
\center{\rotatebox{-0}{\includegraphics[height=1\linewidth,
width=1.3\linewidth]{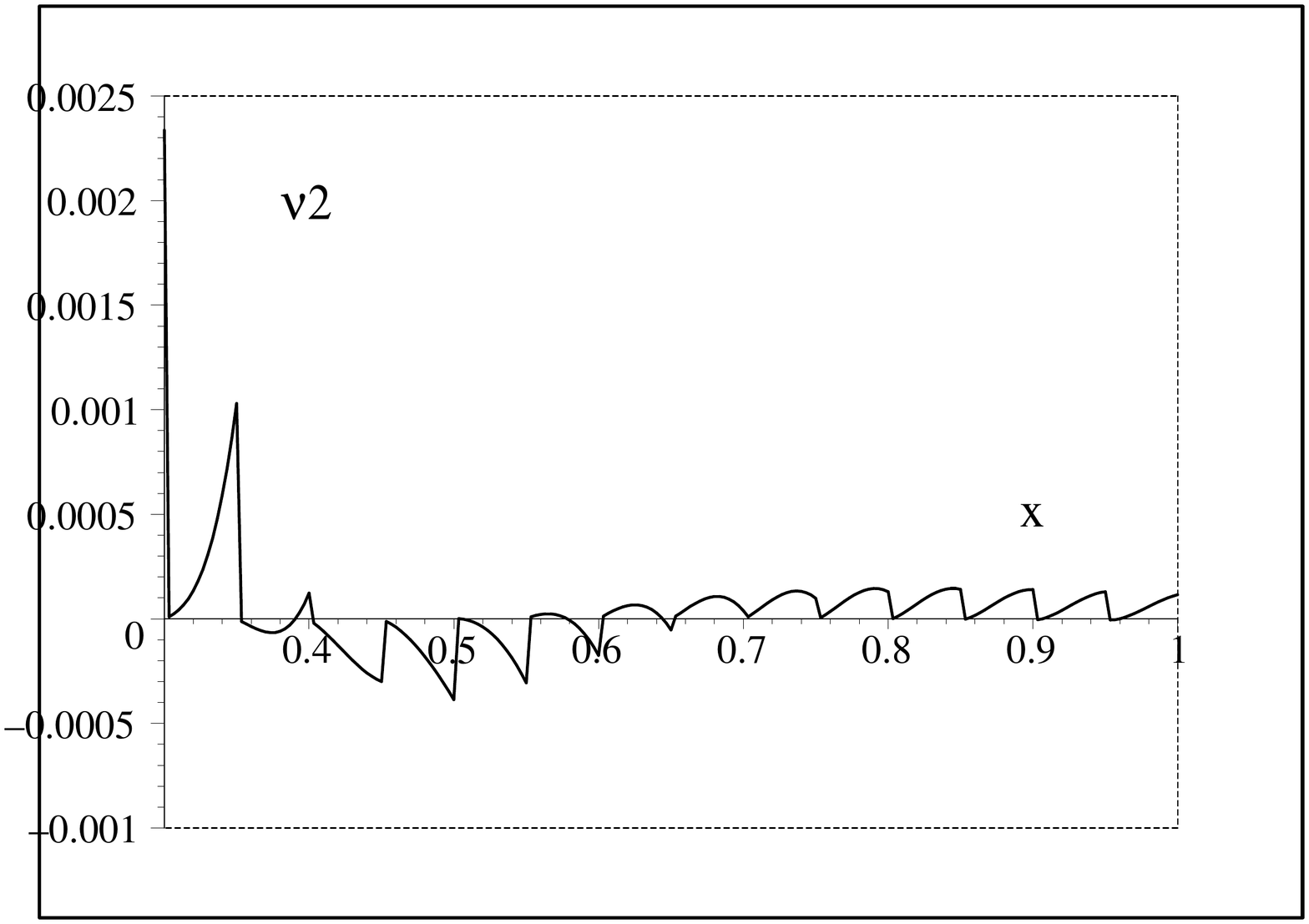}} }
\end{minipage}
\caption{Example 2. FD-method. Graphs of discrepancy.
$\nu_{2}\left(x\right).$}
\end{minipage}
\end{figure}

The graphs on Fig.~2 and Fig.~3 confirm the exponential convergence rate of series (\ref{ZagalniyRyad}) to the exact solution of problem
(\ref{Pr_1}).

{\bf Example 2.}

Let us consider the Cauchy problem
\begin{equation}\label{pr2_1}
    \frac{d}{d x}u\left(x\right)+\left(\frac{1}{\sqrt{x}}+1\right)u^{3}\left(x\right)=\left(\frac{1}{\sqrt{x}}+1\right)\sin\left(2\sqrt{x}+x\right),
\end{equation}
$$ x\in\left[0,\;1\right],\; u\left(0\right)=1.$$

Problem (\ref{pr2_1}) has {a} singularity at the point $x=0:$ $\lim\limits_{x\rightarrow 0}\frac{d}{d x}u\left(x\right)=+\infty.$ Using routines {from} the computer algebra system Maple 12, we {tried} to find the solution of problem (\ref{pr2_1}) either analytically or numerically, {but} all that was in vain. In the present version Maple is unable to {solve} such problems.  The conditions of theorem \ref{FD_Theor} are not fulfilled for this problem, either. Regardless of that, the ideas of  ADM and FD-method are naturally applicable to this problem. However, as it turned out, the ADM is divergent on $\left[0,\;1\right]$ {in this case.}

{To apply} the FD-method {we need to introduce} a grid on the segment $\left[0,\;1\right]$
$$\widehat{\omega}=\left\{0=x_{0}, x_{i}=0.05i,\; i=1,2,\ldots, 20\right\}.$$

The base problem is stated as follows
$$\frac{d}{d x}u^{(0)}\left(x\right)+\left(\frac{1}{\sqrt{x}}+1\right)\left(u^{(0)}\left(x_{i-1}\right)\right)^{2}u^{(0)}\left(x\right)=\left(\frac{1}{\sqrt{x}}+1\right)\sin\left(2\sqrt{x}+x\right), $$
$$x\in\left[x_{i-1},\;x_{i}\right]  u\left(0\right)=1.$$
It admits {the} analytical solution
$$u^{(0)}\left(x\right)=\exp\left(-\left.\left(2\sqrt{\xi}+\xi\right)\right|_{\xi=x_{i-1}}^{\xi=x}\left(u^{(0)}\left(x_{i-1}\right)\right)^{2}\right)u^{(0)}\left(x_{i-1}\right)+$$
$$+\int_{x_{i-1}}^{x}\left(\frac{1}{\sqrt{\xi}}+1\right)\exp\left(\left.\left(2\sqrt{\tau}+\tau\right)\right|^{\tau=\xi}_{\tau=x}\left(u^{(0)}\left(x_{i-1}\right)\right)^{2}\right)\sin\left(2\sqrt{\xi}+\xi\right) d \xi=$$
$$=\exp\left(-\left.\left(2\sqrt{\xi}+\xi\right)\right|_{\xi=x_{i-1}}^{\xi=x}\left(u^{(0)}\left(x_{i-1}\right)\right)^{2}\right)u^{(0)}\left(x_{i-1}\right)+$$
$$+\frac{\left(u^{(0)}\left(x_{i-1}\right)\right)^{2}}{\left(u^{(0)}\left(x_{i-1}\right)\right)^{4}+1}\left[\exp\left(-\left.\left(2\sqrt{\tau}+\tau\right)\right|_{\tau=\xi}^{\tau=x}\left(u^{(0)}\left(x_{i-1}\right)\right)^{2}\right)\right.\times$$
$$\times \left.\left.\left\{\sin\left(2\sqrt{\xi}+\xi\right)-\frac{1}{\left(u^{(0)}\left(x_{i-1}\right)\right)^{2}}\cos\left(2\sqrt{\xi}+\xi\right)\right\}\right|_{\xi=x_{i-1}}^{\xi=x}\right].$$
Similar analytical formulas were obtained for  $u^{(1)}\left(x\right)$ and $u^{(2)}\left(x\right).$ For error control we use the discrepancy
$$\nu_{n}\left(x\right)=\sqrt{x}\left(\frac{d}{d x}\overset{n}{u}\left(x\right)\right)+\left(1+\sqrt{x}\right)\left(\left(\overset{n}{u}\left(x\right)\right)^{3}-\sin\left(2\sqrt{x}+x\right)\right).$$
The results are presented on Fig.~4 -- Fig.~6.
As above we use the notation $\overset{m}{u}\left(x\right)=\sum\limits_{i=0}^{m}u^{(i)}\left(x\right).$
It is easy to seen that the discrepancy is decreasing exponentially even behind the point of singularity $x=0.$


\begin{thebibliography}{00}


\bibitem[Adomian (1984)]{Adomian_1}Adomian G. and Adomian G.E., A Global Method for
Solution of Complex Systems // Mathmatical Modelling, Vol. 5, 1984.

\bibitem[Adomian et al.(1985)Adomian and Rach]{Adomian_2} Adomian G. and Rach, R. On the Solution of
Algebraic Equations by the Decomposition Method // Mathematical
Analisis and Applications, Vol. 105 No. 1, 1985.

\bibitem[Adomian(1994)]{Adomian_3} Adomian G. Solving Frontier Problems of Physics:
The Decomposition Method, Kluwer, Boston, MA, 1994.

\bibitem[Cherruault.(1989)]{Yves_Cherruault} Yves Cherruault. Convergence of Adomian's
Method //Kybernetes, Vol. 18, No. 2, 1989, pp. 31--38.
  
\bibitem[Abbaoui et al.(1995)Abbaoui and Cherrault]{Abbaoui_Cherrault} K. Abbaoui and Y. Cherrault. New Ideas for
  Proving Convergence of Decomposition Methods // Computers Math.
  Applic. Vol. 29, No. 7, pp. 103--108, 1995.
  
\bibitem[Hosseini et al.(2006)Hosseini, Nasabzadeh]{Hosseini_Nasabzadeh} M.M. Hosseini, H. Nasabzadeh. On the
  convergence  of Adomian decomposition method // Applied Mathematics
  and Computation 182 (2006) pp. 536--543.

\bibitem[Abbaoui et al.(2001)Abbaoui, Pujol, Cherruault, Himoun and Grimalt]{Abbaoui_Pujol_Cherruault_Himoun_Grimalt} Abbaoui K., Pujol M. J., Cherruault Y., Himoun N., Grimalt P. A new formulation of Adomian method: convergence result //
Kybernetes.---2001.---\textbf{30},  No. 9-10. --- P.~1183--1191.

\bibitem[El-Kalla (2007)]{Ibrahim_El-Kalla}
Ibrahim L. El-Kalla. Error Analysis Of Adomian Series Solution To A
Class Of Nonlinear Differential Equations // Applied Mathematics
E-Notes, 7(2007), pp. 214-221.

\bibitem[Hashim et al.(2006)Hashim, Noorani, Ahmad, Bakar,
  Ismail and Zakaria]{Hashim_Noorani_Ahmad_Bakar_Ismail_Zakaria} I. Hashim, M.S.M. Noorani, R. Ahmad, S.A. Bakar,
  E.S. Ismail, A.M. Zakaria. Accuracy of the Adomian decomposition
  method applied to the Lorenz system // Chaos, Solitons and
  Fractals, 28 (2006) pp. 1149--1158.
  
  \bibitem[Inc et al.(2005)Inc and Cherruault]{Inc_Cherruault} M. Inc, Y. Cherruault. A reliable method for
  obtaining appriximate solutions of linear and nonlinear
  Volterra-Fredholm integro-differential equations //Kybernetes,
  Vol. 34, No. 7/8, 2005, pp. 1034--1048.
  
  \bibitem[Himoun et al.(2003)Himoun, Abbaoui and Cherruault]{N_Himoun_K_Abbaoui_Y_Cherruault} N. Himoun, K. Abbaoui, Y. Cherruault. New results on Adomian
method //Kybernetes, Vol.~32, No. 4, 2003, pp. 523--539.
  
  \bibitem[Seng et al.(1996)Seng, Abbaoui and Cherruault]{Seng_Abbaoui_Cherruault}
  V. Seng, K. Abbaoui and Y. Cherruault. Adomian's Polynomials for Nonlinear Operators // Math. Comput. Modelling.
  Vol. 24. No. 1. pp. 59--65, 1996.

\bibitem[Abbaoui et al.(1995)Abbaoui, Cherruault and Seng]{Seng_Abbaoui_Cherruault_1}  Abbaoui K., Cherruault Y., Seng V. { Practical formulae for the calculus of multivariable Adomian polynomials} // Math. Comput. Modelling 22 (1995), no. 1, p. 89--93.
  
  \bibitem[Gavrylyuk et al.(2004)Gavrylyuk and Makarov]{Gavriluk_Makarov}
  Strongly Positive Operators and numerical algorithms withaut accuracy saturation/ Gavrylyuk I.P., Makarov V.L. -- Kyiv: Institut of mathematics of NAS of Ukraine, 2004. -- 500 p. (in
Russian)
  
  \bibitem[Makarov(1991)]{Makarov_SLP}
  V.L.Makarov,A functional difference method of arbitrary ordeof accuracy for solving the Sturm—
Liouville problem with piecewise-smooth coeffcients// Dokl.Akad.Nauk SSSR, 320 (1991),pp.34–39(in
Russian).
  
  \bibitem[Gavrilyuk et al.(2007)Gavrilyuk, Klymenko, Makarov and Rossokhata]{GKMR} I. Gavrilyuk, A. Klymenko, V. Makarov and N. Rossokhata, FD-method for eigenvalue problems with nonlinear
  potential //
Ukrainian Mathematical Journal,  ---2007, ---v.~59, pp.--- 14--28

\bibitem[Lazurchak et al.(2008)Lazurchak, Makarov and Sytnyk]{Lazurchak_Makarov_Sytnyk}I.I. Lazurchak, V.L. Makarov, D. Sytnyk. Two-sided Approximations for Nonlinear Operator Equations
 // Computation Methods in Applied Mathematics, Vol.8(2008), No.4, pp.386–392.
\bibitem[Gavrilyuk et al.(2009)Gavrilyuk, Lazurchak, Makarov and Sytnyk]{Gavrilyuk_Lazurchak_Makarov_Sytnik}I.P. Gavrilyuk, I.I. Lazurchak, V.L. Makarov, D. Sytnyk. A Method with a Controllable Exponential Convergence Rate
for Nonlinear Differential Operator Equations // Computation Methods in Applied Mathematics, Vol. 9 (2009), No. 1, pp.
63–78 pp.386–392.

\bibitem[Akhmet(2007)]{Akhmet_PC_argument}  M.U. Akhmet  Integral mainfolds of differential equations with piecewise constant argument of generalized type // Nonlinear analysis 66 (2007) P. 367--383.

\bibitem[Demidivich(1967)]{Demidovich} Demidivich B.P. Lectures on the mathimatical theory of stabulity. -- M.: Nauka. 1967. 472 p.(in
Russian)
  
  \bibitem[Fihtenholts(1966)]{Fihtenholts}
  Fihtenholts G.M. The course of integral and differential calculation. Vol. 2. -- M.: Nauka. 1966 - 800 p.(in
Russian)


 \end{thebibliography}
\end{document}